\pdfoutput=1

\documentclass[11pt]{article}

\pdfoutput=1

\usepackage{enumerate}
\usepackage{amsthm,amsmath,amssymb}
\usepackage{graphicx}
\usepackage{lineno}
\usepackage[colorlinks=true,citecolor=black,linkcolor=black,urlcolor=blue]{hyperref}
\usepackage[english]{babel}
\usepackage{amsfonts}
\usepackage{epsfig, subfigure}
\usepackage{amscd,latexsym}
\usepackage{url}
\usepackage{color}
\usepackage{authblk}
\usepackage{subfigure}

\theoremstyle{plain}
\newtheorem{theorem}{Theorem}
\newtheorem{lemma}[theorem]{Lemma}
\newtheorem{cor}[theorem]{Corollary}
\newtheorem{prop}[theorem]{Proposition}

{\itshape}{\rmfamily}
\newtheorem{remark}[theorem]{Remark}

\begin{document}

\title{Resolving dominating  partitions in graphs}

\author[1]{Carmen Hernando\thanks{Partially supported by projects MTM2015-63791-R (MINECO/FEDER) and Gen.Cat. DGR2017SGR1336, carmen.hernando@upc.edu}}
\author[1]{Merc\`e Mora\thanks{Partially supported by projects MTM2015-63791-R (MINECO/FEDER), Gen.Cat. DGR2017SGR1336 and H2020-MSCA-RISE-2016-734922 CONNECT, merce.mora@upc.edu}}
\author[1]{Ignacio M. Pelayo\thanks{Partially supported by project MINECO MTM2014-60127-P, ignacio.m.pelayo@upc.edu}}

\affil[1]{Departament de Matem\`atiques, Universitat Polit\`ecnica de Catalunya}

\date{}

\maketitle


\begin{abstract}
A  partition $\Pi=\{S_1,\ldots,S_k\}$ of the vertex set of a connected graph $G$ is  called a \emph{resolving  partition}   of $G$ if for every pair of vertices $u$ and $v$, $d(u,S_j)\neq d(v,S_j)$, for some part $S_j$.
The \emph{partition  dimension} $\beta_p(G)$  is the minimum  cardinality of a resolving partition of $G$.
A resolving partition $\Pi$ is called \emph{resolving dominating} if for every vertex $v$ of $G$, $d(v,S_j)=1$, for some  part $S_j$ of $\Pi$.
The \emph{dominating partition dimension} $\eta_p(G)$  is the minimum  cardinality of a resolving dominating partition of $G$.

In this paper we show, among other results,  that $\beta_p(G) \le \eta_p(G) \le \beta_p(G)+1$.
We also characterize all connected graphs of order $n\ge7$ satisfying any of the following conditions:  $\eta_p(G)= n$, $\eta_p(G)= n-1$, $\eta_p(G)= n-2$ and $\beta_p(G) = n-2$.
Finally, we present some tight  Nordhaus-Gaddum bounds for both the  partition dimension $\beta_p(G)$ and the dominating partition dimension $\eta_p(G)$.

\vspace{+.1cm}\noindent \textbf{Keywords:} resolving partition, resolving dominating partition, metric location, resolving domination, partition dimension, dominating partition dimension.

\vspace{+.1cm}\noindent \textbf{AMS subject classification:} 05C12, 05C35, 05C69.
\end{abstract}
\vspace{0.5cm}


\section{Introduction}\label{sec1:intro}

Domination and location in graphs are two important subjects that have received a high attention in the literature,
usually separately, but sometimes also both together.
	These concepts are useful to distinguish the vertices of a graph in terms of distances to a given set of vertices or by considering their neighbors in this set.
{Resolving sets} were introduced independently in the 1970s by Slater~\cite{slater}, as \emph{locating sets}, and  by Harary and Melter~\cite{hararymelter}, whereas {dominating sets} were defined in the 1960s by Ore~\cite{ore}.
Both types of sets have many and varied applications in other areas.
For example, resolving sets have applications in robot navigation~\cite{slater}, combinatorial optimization~\cite{st04}, game theory \cite{goddard}, pharmaceutical chemistry \cite{chartrand} and in other contexts \cite{ chmpp12, chmppsw07}.
On the other hand, dominating sets  are  helpful to design and analyze communication networks \cite{cockayne, sasireta} and to model biological networks \cite{haynes}.

Many variations of location in graphs have since been defined (see  survey  \cite{sz}). For example, in 2000, Chartrand, Salehi and Zhang study the resolvability of graphs in terms of partitions~\cite{ChaSaZh00}, 
	as a generalization of resolving sets when the vertices are classified in different types.
A few years later, 
{resolving dominating sets} were introduced by Brigham, Chartrand, Dutton and  Zhang~\cite{bcdz03} and independently by Henning and Oellermann~\cite{heoe04} as {metric-locating-dominating sets}, combining the usefulness of resolving sets and dominating sets. 
Resolving dominating sets have been further studied in \cite{chmpp13,gohemo17,hmp14}.
In this paper, following the ideas of  these  works, we introduce the \emph{resolving dominating partitions}, 
	as a way for distinguishing the vertices of a graph by using on the one hand  partitions, and on the other hand, both domination and location.

\subsection{Basic terminology}

All the graphs considered are undirected, simple, finite and (unless otherwise stated) connected.
Let $v$ be a vertex of a graph $G$.
The \emph{open neighborhood} of $v$ is $\displaystyle N_G(v)=\{w \in V(G) :vw \in E\}$, and the \emph{closed neighborhood} of $v$ is $N_G[v]=N_G(v)\cup \{v\}$ (we will write $N(v)$ and $N[v]$ if the graph $G$ is clear from the context).
The \emph{degree} of $v$ is $\deg(v)=|N(v)|$.
The minimum degree  (resp. maximum degree) of $G$ is $\delta(G)=\min\{\deg(u):u \in V(G)\}$ (resp. $\Delta(G)=\max\{\deg(u):u \in V(G)\}$).
If $\deg(v)=1$, then $v$ is said to be a  \emph{leaf} of $G$.

The distance between vertices $v,w\in V(G)$ is denoted by $d_G(v,w)$, or $d(v,w)$ if the graph $G$ is clear from the context.
The diameter of $G$ is ${\rm diam}(G) = \max\{d(v,w) : v,w \in V(G)\}$.
The distance between a vertex $v\in V(G)$ and a set of vertices $S\subseteq V(G)$, denoted by  $d(v,S)$, is the minimum of the distances between $v$ and the vertices of $S$, that is, $d(v,S)=\min\{d(v,w):w\in S\}$.

Let $u,v \in V(G)$ be  a pair of vertices such that  $d(u,w)=d(v,w)$ for all $w\in V(G)\setminus\{u,v\}$, i.e.,  such that  either $N(u)=N(v)$ or $N[u]=N[v]$.
In both cases, $u$ and $v$ are said to be \emph{twins}.
Let $W$ be a set of vertices of $G$.
If the vertices of $W$ are pairwise twins, then $W$ is called a \emph{twin set} of $G$.

Let $W\subseteq V(G)$ be a subset of vertices of  $G$.
The  \emph{closed neighborhood} of $W$ is $N[W]=\cup_{v\in W} N[v]$.
The subgraph of $G$ induced by $W$, denoted by $G[W]$, has $W$ as vertex set and $E(G[W]) = \{vw \in E(G) : v \in W,w \in W\}$.

The \emph{complement} of $G$,  denoted by $\overline{G}$, is the  graph on the same vertices as $G$ such that two vertices are adjacent in $\overline{G}$ if and only if they are not adjacent in $G$.
Let $G_1$, $G_2$ be two graphs having disjoint vertex sets.
The (\emph{disjoint}) \emph{union} $G=G_1+G_2$ is the graph such that  $V(G)=V(G_1)\cup V(G_2)$ and $E(G)=E(G_1)\cup E(G_2)$.
The \emph{join} $G=G_1\vee G_2$ is the graph such that
$V(G)=V(G_1)\cup V(G_2)$ and $E(G)=E(G_1)\cup E(G_2)\cup \{uv:u\in V(G_1),v\in V(G_2)\} $.

\subsection{Metric dimension and partition dimension}

A set of vertices $S\subseteq V(G)$ of a graph $G$ is a \emph{resolving set} of $G$
if for every pair of distinct vertices $v,w\in V(G)$,  $d(v,x)\ne d(w,x)$, for some vertex $x\in S$.
The \emph{metric dimension} $\beta(G)$ of $G$ is the minimum cardinality of a resolving set.

Resolving sets were introduced independently in papers \cite{hararymelter} and  \cite{slater} (in this last work they were called \emph{locating sets}), and since then they have been widely investigated
(see \cite{chmppsw07,hmpsw10,st04} and their references).

Let $\Pi=\{S_1,\ldots,S_k\}$ be  a partition of $V(G)$. We denote by $r(u|\Pi)$
the vector of distances between a vertex $u\in V(G)$ and the elements of  $\Pi$, that is, $r(u|\Pi)=(d(u,S_1),\dots ,d(u,S_k))$.
If $u,v\in V(G)$, we say that a part $S_i$ of $\Pi$ \emph{resolves} $u$ and $v$ if $d(u,S_i)\ne d(v,S_i)$.
If $V'\subseteq V(G)$, we say that a part $S_i$ of $\Pi$ \emph{resolves} $V'$ if $S_i$ resolves every pair of vertices  of $V'$.

A partition $\Pi=\{S_1,\ldots,S_k\}$ is called a  \emph{resolving partition} of $G$  if for any pair of distinct vertices $u,v\in V(G)$, $r(u|\Pi)\neq r(v|\Pi)$, that is, if the set $\{u,v\}$ is resolved by some part $S_i$ of $\Pi$.

The \emph{partition  dimension} $\beta_p(G)$ of $G$ is the minimum  cardinality of a resolving partition of $G$.
Resolving partitions were introduced  in \cite{ChaSaZh00}, and further studied in
\cite{chagiha08,fegooe06,gyjakuta14,grstramiwi14,royele14,tom08}. 
In some of these papers the partition dimension of $G$ is denoted by $pd(G)$.
Next, some known results concerning this parameter are shown.

\begin{theorem} [\cite{ChaSaZh00}]\label{mdpd}
Let $G$  be a graph of order $n\ge2$. Then,
\begin{enumerate}[(1)]
\item $\beta_p(G) \le \beta(G)+1$.
\item $\beta_p(G) \le n-{\rm diam}(G)+1$. Moreover, this bound is sharp.
\item $\beta_p(G)=2$ if and only if $G$ is isomorphic to the path $P_n$.
\item $\beta_p(G)=n$ if and only if $G$ is isomorphic to the complete graph $K_n$.
\item If $n\ge6$, then $\beta_p(G)=n-1$ if and only if $G$ is isomorphic to either the star $K_{1,n-1}$, or
	the complete split graph $K_{n-2} \vee \overline{K_2}$,  or the graph $K_1\vee (K_1+K_{n-2})$.
\end{enumerate}
\end{theorem}

\begin{remark}
Notice that the restriction $n\ge6$ of  Theorem \ref{mdpd}(5) is tight, since $\beta_p(C_4)=3$ and $\beta_p(C_4\vee K_1)=4$.
Thus, in \cite{ChaSaZh00}, the condition $n\ge3$ of Theorem 3.3 is incorrect.
\end{remark}

\begin{prop} [\cite{chagiha08}]
Given a pair of integers $a,b$ such that $3 \le a \le b+1$, there exists a graph $G$ with $\beta_p(G)=a$ and $\beta(G)=b$.
\end{prop}

The remaining part of this paper
is organized as follows.
In Section \ref{sec2}, we introduce the dominating partition number $\eta_p(G)$ and show some
basic properties for this new parameter.
Finally, Section \ref{sec3} is devoted to the characterization of all graphs $G$ satisfying any of the following conditions:  $\eta_p(G)= n$, $\eta_p(G)= n-1$, $\eta_p(G)= n-2$ and $\beta_p(G) = n-2$ and to show some tight  Nordhaus-Gaddum bounds for both the  partition dimension $\beta_p(G)$ and the dominating partition dimension $\eta_p(G)$.

\section{Dominating partition dimension}\label{sec2}

A set $D$ of vertices of a graph $G$ is a \emph{dominating set} if $d(v,D)=1$, for every vertex $v\in V(G)\setminus D$.
The \emph{domination number} $\gamma(G)$ is the minimum cardinality of a dominating set.

A set $S\subseteq V(G)$ is a \emph{resolving dominating set}, if it is both resolving and dominating.
The \emph{resolving domination number $\eta(G)$} of $G$ is the minimum cardinality of a resolving dominating set of $G$.
Resolving dominating sets were introduced in  \cite{bcdz03}, and also independently in \cite{heoe04} (in this last work they were called \emph{metric-locating-dominating sets}),  being  further studied in \cite{chmpp13,gohemo17,hhh06,hmp14,mh09,sbca15}.

As a straightforward consequence of these definitions, it holds that (see \cite{chmpp13}):
$$\max\{\gamma(G), \beta(G)\} \leq \eta(G) \leq \gamma(G)+ \beta(G).$$

A partition $\Pi=\{S_1,\ldots,S_k\}$ of $V(G)$ is called \emph{dominating} if for every $v \in V(G)$,   $d(v,S_j)=1$  for some $j \in \{1,\ldots, k\}$.
The \emph{partition domination number} $\gamma_p(G)$ equals the minimum cardinality  of a dominating partition in $G$.

\newpage
\begin{prop}\label{z1}
For any non-trivial graph $G$, $\gamma_p(G)=2$.
\end{prop}
\begin{proof}
Let $S$ be a dominating set of cardinality $\gamma(G)$.
Observe that the partition $\Pi=\{S, V(G)\setminus S\}$  is a dominating partition of $G$.
Hence, $\gamma_p(G)= 2$, since $G$ is non-trivial.
\end{proof}

Let $\Pi=\{S_1,\ldots,S_k\}$ be a partition of the vertex set of a non-trivial graph $G$.
The partition $\Pi$ is called a  \emph{resolving dominating partition} of $G$, \emph{RD-partition} for short,   if it is both resolving and dominating.
The \emph{dominating partition dimension} $\eta_p(G)$ of $G$  is the minimum  cardinality of an RD-partition of $G$.
An RD-partition of cardinality $\eta_p(G)$ is called an \emph{$\eta_p(G)$-partition} of $G$. 

\begin{prop}\label{p2}
If $G$ is  a non-trivial graph, then $\eta_p(G)=2$ if and only if $G$ is isomorphic  to  $K_2$.
\end{prop}
\begin{proof}
Certainly, $\eta_p(K_2)=2$.
Conversely, let $G$ be a graph such that  $\eta_p(G)=2$.
Take an $\eta_p(G)$-partition $\Pi=\{S_1,S_2\}$.
Suppose that for some $i\in\{1,2\}$, $|S_i|\geq 2$.
Assume w.l.o.g. that $i=1$ and take  $u,v \in S_1$.
As $\Pi$ is a dominating partition, $r(u|\Pi)=(0,1)=r(v|\Pi)$, contradicting that $\Pi$ is a resolving partition.
So, $|S_1|=|S_2|=1$  and thus  $G\cong K_2$.
\end{proof}

\begin{prop}
	Let $P_n$ and $C_n$ denote the path and the cycle of order $n$, respectively.
	If $n\geq 3$, then $\eta_p(P_n)=\eta_p(C_n)=3$.
\end{prop}
\begin{proof}
	According to Proposition \ref{p2}, it is sufficient to show, in both cases, the existence of an RD-partition of cardinality $3$.
	Assume that $V=V(P_n)=V(C_n)=\{1,\ldots,n\}$;  $E(P_n)=\{ \{ i,i+1\} : 1\le i< n \}$ and $E(C_n)=E(P_n)\cup \{ \{ 1,n \} \}$. Consider the following sets of vertices:
	
	\begin{center}
		$S_1=\{1\}$,
		$S'_1=\{1,2\}$,
		$S_2=\{i \in V  :  i \hbox{ even} \}$,
		$S'_2=\{i\in V    :  i\not= 2, \hbox{ even}\}$,
		$S_3=\{i \in V     :  i\not= 1, \hbox{ odd}\}$.
	\end{center}
	
	It is  straightforward to check that $\Pi=\{S_1,S_2,S_3\}$ is an RD-partition of $P_n$, and also of  $C_n$ if $n$ is odd, and that $\Pi'=\{S'_1,S'_2,S_3\}$ is an RD-partition of $C_n$, if $n$ is even.
\end{proof}

Next, we show some results relating the dominating partition dimension $\eta_p$ to other parameters such as the resolving domination number  $\eta$,  the partition dimension $\beta_p$, the order and the diameter.

\begin{prop} For any  graph $G$ of order $n\ge2$, $\eta_p(G) \le \eta(G)+1$.
\end{prop}
\begin{proof}
Suppose that $\eta(G)=k$.
Notice that $k\le n-1$, since $n\ge2$.
Let $S=\{u_1,\ldots,u_k\}$ be  a resolving dominating set of $G$.
Then, $\Pi=\{\{u_1\},\ldots,\{u_k\},V(G)\setminus S\}$ is an RD-partition of $G$.
\end{proof}

\begin{lemma}\label{hojcomp}
	Let  $G$ be a graph of order $n\ge3$.
	Let $W\subsetneq V(G)$  be a twin set of cardinality $k\ge2$.
	\begin{enumerate}
		\item[\rm(1)] If $W$ induces an empty graph, then $\eta_p(G)\ge \beta_p(G)\ge k$.
		\item[\rm(2)] If $W$ induces a complete graph, then $\eta_p(G)\ge \beta_p(G)\ge k+1$.
		\item[\rm(3)] If $W$ is a set of leaves, then $\eta_p(G)\ge k+1$.
	\end{enumerate}
\end{lemma}
\begin{proof}
	\begin{enumerate}
		
\item[(1)]
	Let $W$ be a twin set of cardinality $k$.
	Since $d(w_1,v)=d(w_2,v)$ for every $w_1,w_2\in W$ and for every $v\in V(G)\setminus\{w_1,w_2\}$,
	we have that different vertices of $W$ must belong to different parts of any resolving partition.
	Hence, $\eta_p(G)\ge \beta_p(G)\ge k$.
	
	Observe that,
	if $\beta_p(G) = k$, then every part of a resolving partition of cardinality $k$ contains exactly one vertex of $W$.
	
\item[(2)] Suppose that $W$ induces a complete graph and $W\subsetneq V(G)$.
	Since $G$ is connected, there exists a vertex $v$ adjacent to all the vertices of $W$. If $\beta_p(G) = k$ and $\Pi$ is a resolving partition of cardinality $k$, then there is some vertex $w\in W$ such that $v$ and $w$ belong to the same part $S$ of $\Pi$.
	Then, $v$ and $w$ are at distance $1$ from any part of $\Pi$ different from $S$, implying that $r(v|\Pi)=r(w|\Pi)$, a contradiction.
	Therefore, $\beta_p(G) \ge  k+1$.
	
\item[(3)] Assume that $W$ is a twin set of leaves hanging from a vertex $u$.
	Suppose that $\eta_p(G) = k$ and $\Pi$ is an RD-partition of cardinality $k$.
	Then, $\Pi$ is also a resolving partition of cardinality $k$.
	Hence, there is some vertex $w\in W$ such that $u$ and $w$ belong to the same part $S$ of $\Pi$.  But in such a case, $\Pi$ is not a dominating partition, because $w$ is a leaf hanging from $u$. Therefore, $\eta_p(G)\ge k+1$.
	\end{enumerate}
\vspace{-.5cm}\end{proof}

\vspace{.2cm}
\begin{prop}
Given a pair of integers $a,b$ such that $3 \le a \le b+1$, there exists a graph $G$ with $\eta_p(G)=a$ and $\eta(G)=b$.
\end{prop}
\begin{proof}
Let $h=a-2$ and $k=b-a+2$.
Take the caterpillar $G$  of order $n=2k+h$ displayed in Figure \ref{realiz}.
The set $W=\{w_1,\ldots,w_h,u_1\}$ is a twin set of $h+1$ leaves.
Thus, by Lemma~\ref{hojcomp}, we have $\eta_p(G)\ge h+2$.
Now, take the partition
$\Pi=\{\{u_1,\ldots,u_k\},\{v_1,\ldots,v_k\},\{w_1\},\ldots,\{w_h\}\}$.
Clearly, $\Pi$ is both a dominating and a resolving partition.
Hence, $\eta_p(G)= h+2=a$.

To prove that $\eta(G)=b$, note  first that every resolving dominating set $S$ must contain all vertices from the twin set $W$ except at most one.
Observe also that for every $i\in\{1,\dots,k\}$, either $u_i$ or $v_i$ must belong to $S$.
Thus, $\eta(G)\ge h+k=b$.
Now, take the set $S=\{w_1,\ldots,w_h,u_1,\ldots,u_k\}$.
Clearly, $S$ is both dominating and resolving.
Hence, $\eta(G) = h+k=b$.
\end{proof}

\vspace{.4cm}
\begin{figure}[ht]
\begin{center}
\includegraphics[width=0.4\textwidth]{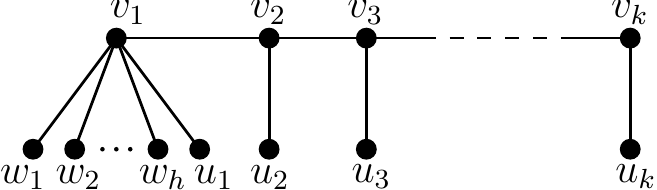}
\caption{Caterpillar $G$  of order $n=2k+h$, $\eta_p(G)=h+2$ and $\eta(G)=h+k$.}
\label{realiz}
\end{center}
\end{figure}

Next, a remarkable double inequality relating both the partition dimension and the dominating partition dimension is shown.

\vspace{.4cm}
\begin{theorem}\label{etabeta}
For any graph  $G$ of order $n\ge 3$, $\beta_p(G) \le \eta_p(G)\le \beta_p(G) +1$.
\end{theorem}

\begin{proof}
The first inequality follows directly  from the definition of RD-partition.
Let $\beta_p(G)=r$ and let $\Pi=\{ S_1,\dots ,S_r\}$  be a resolving partition of $G$.
If $\Pi$ is a dominating partition, then $\eta_p(G)=\beta_p(G)$.
Suppose that $\Pi$ is not a dominating partition.
Let $W=\{ u\in V(G): N[u]\subseteq S_i \textrm{ for some }i\in \{1,\dots ,r\} \}$.
Note that $W\neq \emptyset$, since $\Pi$ is not dominating,  and that $S_i\setminus W\neq \emptyset$ for every $i\in \{ 1,\dots ,r \}$, since $G$ is connected.
In order to show that $\eta_p(G)\le \beta_p(G) +1$, we construct  an RD-partition of cardinality $r+1$.

Let $C_1, \dots , C_s$ be  the connected components  of the subgraph $G[W]$ induced by $W$.
Clearly, for every $i\in\{1,\ldots,s\}$,  all vertices of  $C_i$ belong to the same part of $\Pi$.
Next, we define a subset $W'\subseteq W$ as follows.
If $|V(C_i)|=1$, then add to $W'$ the unique vertex of $C_i$.
If $|V(C_i)|\ge 2$, then consider a 2-coloring of a spanning tree of $C_i$, choose one color and add to $W'$ all  vertices having this color.
Note that, if $V(C_k)\subseteq S_{i_k}$ and a pair of vertices $x,y\in C_k$ are adjacent, then one endpoint of $xy$ is in $W'\cup S_{i_k}$ and the other one belongs to $S_{i_k}\setminus W'$.
Let $\Pi '=\{ S_1',\dots ,S_r', W'\}$, where $S_i'=S_i\setminus W'\subseteq S_i$ for every $1\le i\le r$.
We claim that $\Pi '$ is an RD-partition.

On the one hand, observe that the sets $S_1',\dots ,S_r', W'$ are nonempty by construction.
On the other hand, notice that for every $u\in S_i$, $d(u,S_j)=d(u,w)$ for some vertex $w\in S_j\setminus W$ whenever $i\not= j$.
Indeed, assume to the contrary that $d(u,S_j)=d(u,w)$ and $w\in S_j\cap W$.
Since $w\in W$,  we have $N[w]\subseteq S_j$.
Thus, the vertex $w'$ adjacent to $w$ in a shortest $(u,w)$-path is also in $S_j$, implying that $d(u,S_j) \le d(u,w')<d(u,w)=d(u,S_j)$, a contradiction.
From this last observation, we conclude that $d(u,S_j)=d(u,S_j')$ if $u\in S_i$ and $j\not= i$.

Next, we show that $\Pi '$ is a dominating partition, i.e., that for any $u\in V(G)$,  the vector $r(u|\Pi')$ has at least one component equal to 1.
We distinguish  two cases.

\vspace{.3cm}
\noindent \textbf{Case 1}: $u\in  W'$.
Assume that $u\in S_i$, for some $i\in \{ 1,\dots ,r\}$.
If $u$ belongs to a trivial connected component of $G[W]$, then every neighbor of $u$ is in $S_i'$.
So, $d(u,S_i')=1$.
If $u$ belongs to a non-trivial connected component $C_k$ of $G[W]$, then any neighbor of $u$ with different color in the spanning tree of $C_k$ considered in the construction of $W'$ belongs to $S_i'$. So, $d(u,S_i')=1$.

\vspace{.3cm}
\noindent \textbf{Case 2}: $u\in S_i'$, for some $i\in \{ 1,\dots ,r\}$.
If $u\notin W$, as  $u\in S_i'\setminus W= S_i\setminus W$, then $u$ has a neighbor $v$ in some $S_j$ with $j\not= i$.
Therefore, $d(u,S_j')=1$ if $v\in S_j'$, and $d(u,W')=1$ if $v\in W'$.
If $u\in W$, then $u$ belongs to a non-trivial connected component of $G[W]$ and, by construction of $W'$,  $u$ has a neighbor in $W'$. Thus, $d(u,W')=1$.

Finally, we show that $\Pi '$ is a resolving partition, i.e., that $r(u|\Pi')\not= r(v|\Pi')$ for every pair of distinct vertices $u,v\in V(G)$ belonging to the same part of $\Pi'$.
We distinguish  two cases.

\vspace{.3cm}
\noindent \textbf{Case 1}:
$u,v\in S_i'$  for some $i\in \{ 1,\dots ,r\}$.
In such a case,  $u,v\in S_i$.
Since $\Pi$ is a resolving partition, $d(u,S_j)\not= d(v,S_j)$ for some $j\not= i$.
Using the observation above, we have that $d(u,S_j')=d(u,S_j)\not= d(v,S_j)=d(v,S_j')$ for some $j \not= i$.
Therefore, $r(u|\Pi')\neq r(v|\Pi')$.

\vspace{.3cm}
\noindent \textbf{Case 2}:  $u,v\in W'$.
If $u,v\in S_i$ for some $i\in \{ 1,\dots ,r\}$, then proceeding as in the previous case, we have  $r(u|\Pi')\neq r(v|\Pi')$.
Suppose thus  that $u\in S_i$ and $v\in S_j$ with $i\not= j$.
Notice that  $d(u,S_i')=1$ and  $N[v]\subseteq S_j$ because $v\in S_j$ and $v\in W'\subseteq W$.
Thus, $d(v,S_i)\ge 2$, and so $d(v,S_i')=d(v,S_i)\ge 2$.
Finally, from $d(u,S_i')\not= d(v,S_i')$ we get that $r(u|\Pi')\neq r(v|\Pi')$.
\vspace{-.01cm}\end{proof}

The following result is a direct consequence of Theorem \ref{mdpd}(2)  and Theorem \ref{etabeta}.

\vspace{.2cm}
\begin{cor}\label{n-d+2}
If $G$  is a graph of order $n\ge3$, then $\eta_p(G) \leq n-{\rm diam}(G)+2$.
Moreover, this bound is sharp, and is attained, among others, by $P_n$ and $K_{1,n-1}$.
\end{cor}

\vspace{.2cm}
\begin{prop}\label{newbound}
If $G$ is a graph of order $n\ge3$ and diameter $d$ such that $\eta_p(G)=k$,
then $n\le k\, (d^{k-1}-(d-1)^{k-1})$.
\end{prop}

\begin{proof}  Let $\Pi=\{ S_1,\dots ,S_k\}$ be an RD-partition.
If $u\in S_i$, then the $i$-th component of $r(u|\Pi)$ is $0$, any other component is a value from $\{ 1,2,\dots , d\}$ and at least one component must be $1$.
Since there are $d^{k-1}-(d-1)^{k-1}$ such $k$-tuples, we have that  $|S_i|\le d^{k-1}-(d-1)^{k-1}$, and therefore,
$\displaystyle n\le \sum_{i=1}^k |S_i|\le k(d^{k-1}-(d-1)^{k-1})$.
\end{proof}

\section{Extremal graphs}\label{sec3}

In \cite{ChaSaZh00,tom08}, all graphs  of order $n\ge9$ satisfying $\beta_p(G)=n$, $\beta_p(G)=n-1$ and $\beta_p(G)=n-2$ were characterized. 
This section is devoted to approach the same problems for the dominating partition dimension $\eta_p(G)$.
To this end, we prove a pair of technical lemmas.

\begin{lemma}\label{lem.fusion}
Let $k\ge 2$ be an integer.
Let $G$ be a graph of order $n$
containing a vertex $u$ of degree $d$.
If  $n\ge 2k+1$ and $k\le d \le n-k-1$, then $\eta_p (G)\le n-k$.
\end{lemma}
\begin{proof}
Let $N(u)=\{ x_1,\dots , x_k,\dots, x_d\}$.
Let $L$ be the set containing  all   leaves at distance 2 from $u$
and let  $C$ be the set containing both  all non-leaves at distance $2$ and all vertices at distance at least $3$ from $u$, i.e., $C=V(G)\setminus (N[u]\cup L)$.
Assume that  $|L|=l$ and $|C|=c$ and  observe that $l +c =n-d-1\ge k$.

If $c \ge k$, then take the partition
$\Pi =\{ \{ x_1,y_1\}, \dots ,\{ x_k,y_k\} \}\cup \{ \{ z \} : z\notin \{ x_1, \dots , x_k,y_1, \dots , y_k\} \}$, where $y_1,\dots ,y_k\in C$.
Notice that $\Pi$ is a resolving partition since, for every $i\in\{1,\dots ,k\}$,  $\{u \}$ resolves the pair $x_i,y_i$, because $d(u,x_i)=1<2\le d(u,y_i)$.
Furthermore,  for every $i\in \{1,\dots ,k\}$, vertex  $x_i$ is adjacent to $u$ and  vertex $y_i$ is adjacent to a vertex different from $x_i$, because in the case $y_i$ has degree 1, its neighbor does not belong to $N(u)$ by definition of $C$.
So, $\Pi$ is also a dominating partition and thus $\eta_p(G)\le n-k$.

Now, assume $c < k$. Let $h =k-c$ and observe that $1 \le h \le l$ since $l+c \ge k$.
First, we seek if it is possible to pair $h$ vertices of $L$ with $h$ vertices of $N(u)$ satisfying that each pair is formed by non-adjacent vertices.
Observe that this is equivalent to finding a matching $M$ that saturates a subset $L'$ of $L$  of cardinality $h$  in the bipartite graph $H$ defined as follows:
$N(u)$ and $L$ are its partite sets,  and  if $x_i\in N(u)$ and $z\in L$, then $x_iz\in E(H)$ if and only if  $x_iz\notin E(G)$.
So, the degree  in $H$ of a vertex $z\in L$ is $\deg_H(z)=d-1$.
For every nonempty set $W\subseteq L$ with $|W|\le k-1$, we have $|W|\le k-1\le d-1\le |N_H(W)|$,
and for $W\subseteq L$ with $|W|=k$ we have
$|W|\le |N_H(W)|$ whenever $d\ge k+1$ or $|N_H(W)|\ge k$.
Therefore, according to Hall's Theorem, there exists a matching $M$ saturating a subset $L'$ of $L$ of cardinality $h$, except for the case $h=k=d$, provided that $|N_H(W)|<k$ for every subset $W\subseteq L$ with $|W|=k$.
Let $M$ be such a matching, whenever it exists.
We distinguish two cases.

\vspace{.2cm}
\noindent \textbf{Case 1}: $h<k$.
Consider the partition $\Pi$ formed by the $h$ pairs of the matching $M$,
$c$ pairs formed by pairing the  vertices in $C$ with $c$ vertices in $N(u)$ not used in the matching $M$, and a part for each one of the remaining vertices formed only by the vertex itself.
Part $\{ u \}$ resolves each part of cardinality 2 and, by construction, $\Pi$ is dominating.
Thus, $\Pi$ is an RD-partition, implying that $\eta_p(G)\le n-k$.

\vspace{.2cm}
\noindent \textbf{Case 2}: $h=k$.
In such a case, $c=0$ (i.e., $L=V(G)\setminus N[u]$).
If $d> k$, then consider the partition $\Pi$ formed by the $k$ pairs of the matching $M$ and a part for each one of the remaining vertices formed only by the vertex itself.
As in the preceding case, it can be shown that $\Pi$ is an RD-partition, and so $\eta_p(G)\le n-k$.

If $d=h=k$ and there is a subset $W$ of $L$ of cardinality $k$ with $|N_H(W)|\ge k$, then  there exists a matching $M$ between the vertices of $W$ and the vertices of $N(u)$.
Consider the partition $\Pi$ formed by the $k$ pairs of the matching $M$ and a part for each one of the remaining vertices formed only by the vertex itself.
As in the preceding case, it can be shown that $\Pi$ is an RD-partition, and so $\eta_p(G)\le n-k$.

Finally, if $d=h=k$ and there is no subset $W$ of $L$ of cardinality $k$ with $|N_H(W)|\ge k$, then
all vertices of $L$ are leaves hanging from the same vertex of $N(u)$.
We may assume without loss of generality that all vertices in $L$ are adjacent to $x_1$.
Let $y_1,\dots,y_k\in L$ (they exist because $n\ge 2k+1$).
Consider the partition
$\Pi =\{ \{ u,y_1\}, \{ x_2,y_2\}, \dots ,\{ x_k,y_k\} \}\cup \{ \{ z \} : z\notin \{ u, x_2,\dots ,x_k,y_1, \dots ,y_k\} \}$
(see Figure~\ref{fig_Fusion}).
Notice that $\Pi$ is a resolving partition since, for every $i\in\{2,\dots ,k\}$, $P_1=\{u, y_1\}$ resolves the pair $x_i,y_i$ because $d(x_i,P_1)=d(x_i,u)=1<2= d(y_i,P_1)$; and $P_2=\{x_2,y_2\}$ resolves the pair $u,y_1$, because $d(u,P_2)=d(u, x_2)=1<2= d(y_1,P_2)$.
Besides, every vertex has a neighbor in another part by construction.
Thus, $\Pi$ is an RD-partition, implying that $\eta_p(G)\le n-k$.
\begin{figure}[ht]
\begin{center}
\includegraphics[width=0.28\textwidth]{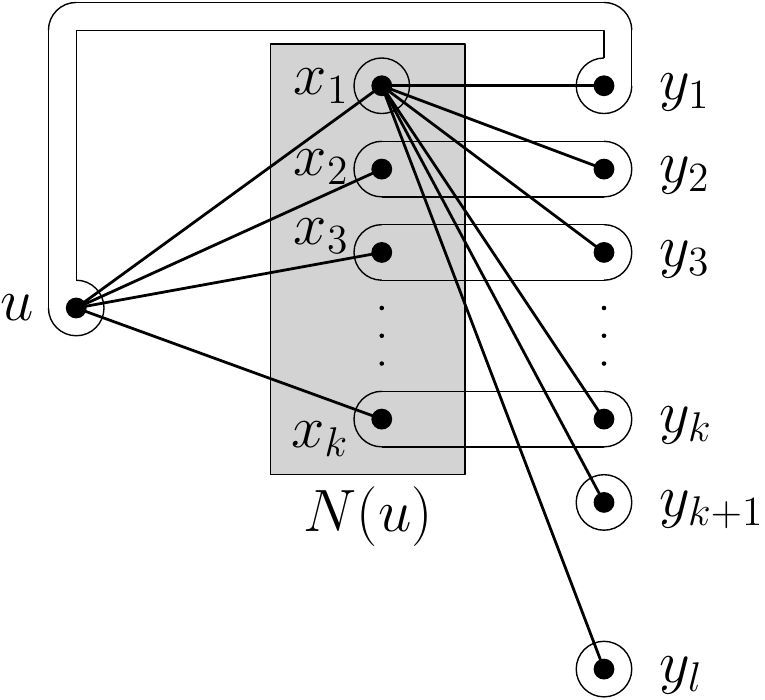}
\caption{An RD-partition of cardinality $n-k$. There may be edges joining vertices of $N(u)$.}
\label{fig_Fusion}
\end{center}
\end{figure}
\end{proof}

\begin{lemma}\label{d3etan-2}
Let $G$ be a graph order $n$.
\begin{enumerate}

\item[\rm(1)] If $n\ge5$ and ${\rm diam}(G)\ge 3$, then $\eta_p(G)\le n-2$.

\item[\rm(2)] If $n\ge7$ and ${\rm diam}(G)\ge 4$, then $\eta_p(G)\le n-3$.

\end{enumerate}
\end{lemma}
\begin{proof}
{\rm(1)}
Let ${\rm diam}(G)=d$.
If $d\ge 4$, then according to Corollary \ref{n-d+2}, $\eta_p(G) \le n-d+2 \le n-2$.
Assume thus that $d=3$ and  take  a vertex $u$ of eccentricity $ecc(u)=3$.
If $u$ is not a leaf, then $2 \le \deg(u) \le n-3$ and, by Lemma \ref{lem.fusion}, $\eta_p(G)\le n-2$.
If $u$ is a leaf, then consider the sets  $D_i=\{v \ | \ d(u,v)=i\}$, $i\in\{1, 2, 3\}$.
Take  $x_i \in D_i$, $i\in\{1,2, 3\}$ such that $\{ux_1,x_1x_2,x_2x_3\} \subseteq E(G)$.
We distinguish cases depending on the cardinality of $D_2$.

\vspace{.1cm}
\noindent \textbf{Case 1}: $|D_2|\ge2$.
 Take a vertex $y_2\in D_2- x_2$.
 Note that $x_1y_2 \in E(G)$, since $u$ is a leaf.
Take  the partition:
$$\Pi =
\{ \{x_1, x_2\},  \{x_3,y_2 \} \cup
\{ \{ z \} : z\neq  x_1, x_2,x_3,y_2 \}.
$$
Clearly, $\Pi$ is an RD-partition of $G$  of cardinality $n-2$.
Thus,  $\eta_p(G)\leq n-2$.

\vspace{.1cm}
\noindent \textbf{Case 2}: $|D_2|=1$.
Notice that $|D_3|\ge2$ since $n\ge5$.
Take a vertex $y_3\in D_3- x_3$.
Observe that $x_2y_3\in E(G)$.
Take  the partition:
$$\Pi =
\{ \{x_1, x_2\},  \{u,y_3 \} \cup
\{ \{ z \} : z\neq  u, x_1, x_2,y_3 \}.
$$
Clearly, $\Pi$ is an RD-partition of $G$  of cardinality $n-2$.
Thus,  $\eta_p(G)\leq n-2$.

\item[(2)] If $d\ge 5$, then according to Corollary \ref{n-d+2}, $\eta_p(G) \le n-d+2 \le n-3$.
Assume thus that $d=4$ and  take  a vertex $u$ of eccentricity of $ecc(u)=4$.
Notice that  $\deg(u)\leq n-4$ and hence, according to  Lemma \ref{lem.fusion} (case $k=3$),  $\eta_p(G)\le n-3$
whenever $\deg(u)\geq 3$.
Suppose finally that $1\le \deg(u) \le 2$ and consider the sets  $D_i=\{v \ | \ d(u,v)=i\}$, $i\in\{1, 2,3, 4\}$.
Notice that $1 \le |D_1| \le 2$.
Take  $x_i \in D_i$, $i\in\{1, 2,3, 4\}$ such that $\{ux_1,x_1x_2,x_2x_3,x_3x_4\}\subseteq E(G)$.
We distinguish cases depending on the cardinality of $D_1$ and $D_2$.

\vspace{.2cm}
\noindent \textbf{Case 1}: $|D_1|=2$.
Take a vertex $y_1\in D_1- x_1$.
Take  the partition:
$$\Pi =
\{ \{u, x_1\},  \{x_2, x_3\}, \{x_4,y_1 \} \cup
\{ \{ z \} : z\neq  u,x_1, x_2,x_3,x_4,y_1 \}.
$$
Clearly, $\Pi$ is an RD-partition of $G$  of cardinality $n-3$.
Thus,  $\eta_p(G)\leq n-3$.

\vspace{.2cm}
\noindent \textbf{Case 2}: $|D_1|=1$ and $|D_2|\ge2$.
Take a vertex $y_2\in D_2- x_2$.
Take  the partition:
$$\Pi =
\{ \{u, x_4\},  \{x_1, x_2\}, \{x_3,y_2 \} \cup
\{ \{ z \} : z\neq  u,x_1, x_2,x_3,x_4,y_2 \}.
$$
Clearly, $\Pi$ is an RD-partition of $G$  of cardinality $n-3$.
Thus,  $\eta_p(G)\leq n-3$.

\vspace{.2cm}
\noindent \textbf{Case 3}: $|D_1|=1$, $|D_2|=1$ and $|D_3|\ge2$.
Take a pair of  vertices $y_3,w\in D_3\cup D_4 \setminus \{x_3,x_4\}$ such that $y_3\in D_3$.
Take  the partition:
$$\Pi =
\{ \{x_1, w\},  \{x_2, x_3\}, \{x_4,y_3 \} \cup
\{ \{ z \} : z\neq  x_1, x_2,x_3,x_4,y_3,w \}.
$$
Clearly, $\Pi$ is an RD-partition of $G$  of cardinality $n-3$.
Thus,  $\eta_p(G)\leq n-3$.

\vspace{.2cm}
\noindent \textbf{Case 4}: $|D_1|=1$, $|D_2|=1$ and $|D_3|=1$.
Take a pair of  vertices $y_4,w_4\in D_4 -x_4$.
Note that $\{x_3y_4,x_3w_4\}\subseteq E(G)$.
Take  the partition:
$$\Pi =
\{ \{u, y_4\},  \{x_1 w_4\}, \{x_2,x_3 \} \cup
\{ \{ z \} : z\neq  u, x_1, x_2,x_3,x_4,y_4,w_4 \}.
$$
Clearly, $\Pi$ is an RD-partition of $G$  of cardinality $n-3$.
Thus,  $\eta_p(G)\leq n-3$.
\end{proof}

In \cite{ChaSaZh00}, all graphs  of order $n$ satisfying $n-1 \le \beta_p  \le n$ were  characterized (see Theorem  \ref{mdpd}).
We display a similar result for the dominating partition dimension $\eta_p$.

\begin{theorem}\label{etann-1}
If $G$  is  a graph of order $n\ge6$, then

\begin{enumerate}
\item[\rm(1)] $\eta_p(G)=n$ if and only if $G$ is isomorphic to either the complete graph $K_n$ or the star $K_{1,n-1}$.
\item[\rm(2)]
$\eta_p(G)=n-1$ if and only if $G$ is isomorphic to either the complete split graph $K_{n-2} \vee \overline{K_2}$,  or the graph $K_1\vee (K_1+K_{n-2})$.
\end{enumerate}
\end{theorem}
\begin{proof}
\begin{enumerate}

\item[(1)] According to Theorem \ref{etabeta}, if $\eta_p(G)=n$  then  $n-1\leq \beta_p(G) \leq n$.
By direct inspection on graphs with $\beta_p(G)=n$  and $\beta_p(G)=n-1$ (see Theorem \ref{mdpd})  the stated result is derived.

\item[(2)]  
	It is a routine exercise to check that  $\eta_p(K_{n-2} \vee \overline{K_2})=\eta_p(K_1\vee (K_1+K_{n-2}))=n-1$.
Conversely, let $G$ be a graph such that $\eta_p(G)=n-1$.
By Lemma \ref{d3etan-2}(1), ${\rm diam}(G)=2$, since $G\not\cong K_n$.
Take a pair of vertices $u,v$ such that $d(u,v)=2$.
By Lemma \ref{lem.fusion}  (case $k=2$), $\deg(u),\deg(v)\in \{1,n-2\}$.
We distinguish three cases.

\noindent \textbf{Case 1}: $\deg(u)=\deg(v)=1$.
Let $w$ be the vertex such that $N(u)=N(v)=\{w\}$.
By Lemma \ref{lem.fusion}, the rest of vertices of $G$ have degree 1, as they are not adjacent neither to $u$ nor to $v$.
Hence, all vertices of $G$ other than vertex $w$ are leaves hanging from $w$, i.e., $G\cong K_{1,n-1}$, a contradiction.

\noindent \textbf{Case 2}: $\deg(u)=\deg(v)=n-2$.
In this case, $N(u)=N(v)=V(G) \setminus \{u,v\}=W$ and for all vertex $z \in W$, $\deg(z)\geq 2$.
Then, by  Lemma \ref{lem.fusion}  (case $k=2$), $\deg(z) \in \{n-2, n-1\}$.

      If  $\deg(z)= n-1$ for all $z \in W$,  then $G$ is isomorphic to  the complete split graph $K_{n-2} \vee \overline{K_2}$.
			
      If there is a vertex $t \in W$ such that $\deg(t)= n-2$, then let $s \in W$ be the vertex that is not adjacent to $t$.
      Observe that both $t$ and $s$ are adjacent to any other vertex of $W$.
      If $a,b\in W \setminus \{s,t\}$, then
      $\Pi =
\{ \{ u, a \}, \{ s, b \}  \} \cup
\{ \{ z \} : z\neq a, b, u, s \}
$
is an RD-partition, and thus $\eta_p (G)\le n-2$.

\noindent \textbf{Case 3}: $\deg(u)=1$ and $\deg(u)=n-2$.
Let $w$ be the vertex adjacent to $u$.
	Since the diameter is $2$, every vertex $t \notin \{u, w, v\}$ is adjacent both to $w$ and $v$.
	In particular, for all vertex $t \notin \{u, w, v\}$, $\deg(t)\geq 2$ and, by Lemma \ref{lem.fusion} (case $k=2$), $\deg(t)=n-2$ and then $G$ is isomorphic to the graph $K_1\vee (K_1+K_{n-2})$.
\end{enumerate}
\vspace{-.5cm}\end{proof}

Next, we characterize those graphs with $\eta_p(G)=n-2$.
Concretely, we prove  that, for every integer $n\ge7$, a graph  of order $n$ satisfies $\eta_p(G)=n-2$ if and only if it belongs to the family $\Lambda_n=\{H_1,\dots ,H_{17}\}$  (see Figure \ref{17grafos}).

\vspace{.4cm}
\begin{figure}[hbt]
	\begin{center}
		\includegraphics[width=0.88\textwidth]{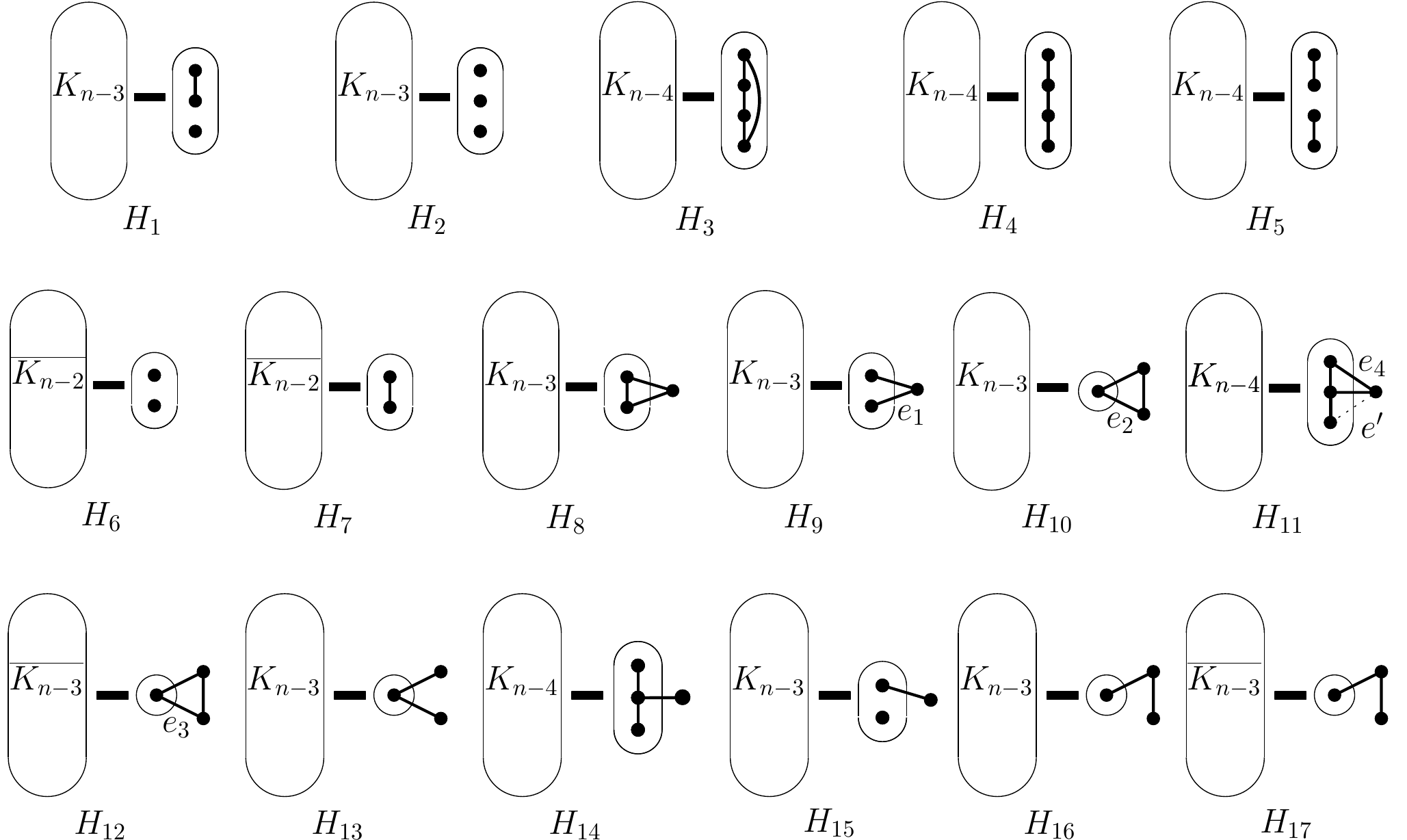}
	\end{center}
	
	\begin{center}
		\begin{tabular}{lll}
			\hline
			&\\
			$ H_1\cong K_{n-3}\vee (K_2+K_1)$    &
			$ H_2\cong K_{n-3}\vee \overline{K_{3}} $ &
			$ H_{3}\cong  K_{n-4}\vee C_4  $  \\
			$ H_{4}\cong K_{n-4}\vee P_4  $  &
			$ H_{5}\cong  K_{n-4}\vee 2\, K_2  $ &
			$ H_6 \cong  K_{2,n-2}$ \\
			$ H_7\cong \overline{K_{n-2}} \vee K_2$ &
			$ H_8\cong (K_{n-3}+K_1)\vee K_2 $ &
			$ H_9\cong  (K_{n-3}+K_1)\vee \overline{K_2} $  \\
			$ H_{10}\cong  (K_{n-3} + K_2) \vee K_1 $ &
			$ H_{11}\cong  (K_{n-4} + K_1) \vee P_3 - e'$ &
			$ H_{12}\cong (\overline{K_{n-3}}+ K_2) \vee K_1 $ \\
			$ H_{13}\cong  (K_{n-3}+\overline{K_2})\vee K_1 $ &
			$ H_{14}\cong H_{11} - e_4$ &
			$ H_{15}\cong  H_9 - e_1 $  \\
			$ H_{16}\cong H_{10} - e_2 $ &
			$ H_{17}\cong  H_{12} - e_3 $ \\
			&\\
			\hline
		\end{tabular}
		\caption{The family $\Lambda_n$ of all graphs of order $n\ge7$ such that $\eta_p(G)=n-2$.}
		\label{17grafos}       
	\end{center}
\end{figure}

\begin{prop}\label{propida}
	
	If $G\in\Lambda_n=\{  H_1,\dots ,H_{17}\}$, then $\eta_p(G)=n-2$.
	Moreover, if $G\in \Lambda_n\setminus \{H_{12},H_{17}\}$, then $\beta_p(G)=n-2$.
\end{prop}

\begin{proof}
	
	According to Theorem \ref{etann-1}, for every graph $H_i\in \Lambda_n$, $\beta_p(G) \le \eta_p(G)\le n-2$.
	Thus, it is enough  to check that, for every graph $H_i\in \Lambda_n$, $\eta_p(H_i)\geq n-2$, and also that if $i\not\in \{12,17\}$, then $\beta_p(H_{i})\geq n-2$.
	
	\vspace{.2cm}
	\noindent \textbf{Case 1}: If $G \in \{H_6, H_7\}$, then it contains a twin set  $W$ of cardinality $n-2$ (see Figure~\ref{17grafos})  and thus, by Lemma \ref{hojcomp}, $\eta_p(G)\ge \beta_p(G)\geq n-2$.
	
	\vspace{.2cm}
	\noindent \textbf{Case 2}: If $G \in \{H_1, H_{2}, H_{8}, H_{9}, H_{10}, H_{13}, H_{15}, H_{16}\}$, then there exists a set of vertices  $W$ of $n-3$ vertices of $G$ such that $W$ induces a complete graph (see Figure~\ref{17grafos}), and thus, according to Lemma \ref{hojcomp}, $\eta_p(G)\ge \beta_p(G)\ge (n-3)+1=n-2$.
	
	\vspace{.2cm}
	\noindent \textbf{Case 3}: If $G \in \{H_{12}, H_{17}\}$, then $G$ is a graph with a twin set of $n-3$ leaves  (see Figure~\ref{17grafos}) and, by Lemma \ref{hojcomp},
	$\beta_p(G)\ge n-3$ and $\eta_p(G)\ge (n-3)+1=n-2$.
	
	\vspace{.2cm}
	\noindent \textbf{Case 4}: If $G \in \{H_{3}, H_{4}, H_{5}, H_{11}, H_{14}\}$, then there exists a twin set $W$ of cardinality $n-4$ that $W$ induces a complete graph (see Figure~\ref{17grafos}), and thus, by Lemma \ref{hojcomp}, $\eta_p(G)\ge \beta_p(G)\ge (n-4)+1=n-3$.
	Suppose that there exists a resolving partition  $\Pi=\{S_1, \dots, S_{n-3}\}$ of cardinality $n-3$.
	Assume that $W=\{w_1, \dots, w_{n-4}\}$ and $w_i\in S_i$, for every $i\in\{1,\ldots,n-4\}$, so that $S_{n-3}\cap W=\emptyset$.
	Notice also that all these graphs have diameter $2$. We distinguish two cases.
	
	\vspace{.2cm}
	\noindent \textbf{Case 4.1}:
	If $G \in \{H_{3},H_{4},H_{5}\}$, then $N[W]=V(G)$ and
	$|V(G) \setminus W|=4$.
	Clearly,  $\vert S_{n-3}\vert=1$, since $r(z|\Pi)=(1, \dots,1,0)$ for every $z\in S_{n-3}$.
	Notice also that $\vert S_{i}\vert\leq2$ for  $i\in \{1, \dots ,n-4\}$, since for every $x\in S_i$ we have
	$r(x|\Pi)=(1,\ldots,1,\overset{i)}{0},1,\ldots,1,h)$, with $h\in\{1,2\}$.
	Hence, there must be exactly three sets of $\Pi$ of cardinality 2 and we can suppose without loss of generality that $S_1=\{w_1,x\}$, $S_2=\{w_2,y\}$, $S_3=\{w_3,z\}$ and $S_{n-3}=\{t\}$,
	where $\{x,y,z,t\}=V(G)\setminus W$. We know that $d(t,w_1)=d(t,w_2)=d(t,w_3)=1$,
	hence $d(t,x)=d(t,y)=d(t,z)=2$, a contradiction, because there is no vertex satisfying this condition in $V(G) \setminus W$.
	
	\vspace{.2cm}	
	\noindent \textbf{Case 4.2}:
	If $G \in \{H_{11},H_{14}\}$, then $\vert N[W] \setminus W\vert =3$.
	We may assume $N[W] \setminus W=\{a,b,c\}$
	and  $V(G)\setminus N[W]=\{z\}$ with $d(a,b)=d(b,c)=1$, $d(b,z)=1$ and $d(c,z)=2$ in both graphs.
	Notice that $S_{n-3}$ has as most one vertex from
	$\{a,b,c\}$,
	since  $r(x|\Pi)=(1, \dots,1,0)$ whenever $x\in\{a,b,c\}\cap S_{n-3}$.
	Moreover, $b\notin S_{n-3}$,
	because if $b\in S_{n-3}$, then $a\notin S_{n-3}$ so that $a\in S_i$, for some $i\in \{ 1,\dots ,n-4\}$, and then $r(a|\Pi)=r(w_i|\Pi)=(1,\dots,1,\overset{i)}0,1,\dots,1,1)$, a contradiction.
	So, we can assume without loss of generality that $\{w_1,b\}\subseteq S_1$.
	Thus, $S_{n-3}=\{  z \}$, otherwise $a$ or $c$ should belong to $S_{n-3}$, so that  $r(w_1|\Pi)=r(b|\Pi)=(0,1,\dots ,1,1)$, a contradiction. Hence $c\in S_j$, for some $j\in \{ 1,\dots,n-4\}$,
	but then $r(w_j|\Pi)= r(c|\Pi)=(1,\dots , 1,\overset{j)}0,1\dots ,1,2)$, a contradiction.
\end{proof}


\vspace{.2cm}
The remainder of this section is devoted to showing that these 17 graph families are the only ones satisfying  $\eta_p(G)=n-2$.

First, note that as a direct consequence of Lemma \ref{d3etan-2}(2) the following result is derived.

\begin{cor}
	If $G$ is a graph with $\eta_p(G)= n-2$, then  $2 \le {\rm diam}(G) \le 3$.
\end{cor}\label{d2d3}

\subsection{Case diameter 2}

Let $G$ be a graph such that  $\eta_p(G)= n-2$ and ${\rm diam}(G)=2$.
We distinguish two cases depending whether $\delta(G)\ge n-3$ or $\delta(G)\le n-4$.
To approach the  first case (notice that the restriction ${\rm diam}(G)=2$ is redundant) we need the following technical lemma.

\begin{lemma}\label{lemma22}
	Let $G$ be a graph of order  $n\ge 7$ and   minimum degree $\delta (G)$ at least $n-3$.
	If $G$ contains at most $n-5$ vertices of degree $n-1$, then $\eta_p (G)\le n-3$.
\end{lemma}
\begin{proof}
	Observe that the complement $\overline{G}$ of $G$ is a (non-necessarily connected) graph with vertices of degree $0$, $1$ or $2$.
	Thus, the components of $\overline{G}$ are either isolated vertices,
	or paths of order at least 2,
	or cycles of order at least 3.
	By hypothesis, $G$ has at most $n-5$ vertices of degree $n-1$, therefore $\overline{G}$ has at least $5$ vertices of degree $1$ or $2$.
	We distinguish three cases.
	
	\vspace{.4cm}\noindent \textbf{Case 1}: $\overline{G}$ has only one non-trivial component.
	In such a case, $\overline{G}$ has al least a (non-necessarily induced)  subgraph isomorphic to $P_5$.
	Let $x_1$, $x_2$, $x_3$, $x_4$ and $x_5$ be the vertices of this path, where $x_ix_{i+1}\in E(\overline{G})$ for $i=1,2,3,4$.
	Let $z\notin \{  x_1,x_2,x_3,x_4,x_5 \}$.
	Consider the partition:
	$$\Pi =
	\{\{ x_1,x_3,x_5,z\} \} \cup
	\{ \{ v \} : v\notin  \{  x_1,x_3,x_5, z \} \}.
	$$
	We claim that $\Pi$ is an RD-partition of $G$  (see Figure~\ref{fig1} (a)).
	Indeed,
	if $S_1=\{ x_2 \}$ and $S_2=\{ x_4 \}$, then
	$r(x_1|\Pi)=(2,1,\dots)$, $r(x_3|\Pi)=(2,2,\dots)$, $r(x_5|\Pi)=(1,2,\dots)$, $r(z|\Pi)=(1,1,\dots)$.
	Moreover, $x_3$ is adjacent in $G$ to any vertex $w\notin  \{  x_1,x_2,x_3,x_4,x_5, z \}$, that exists because the order of $G$ is at least 7.
	Therefore, $\Pi$ is an RD-partition of $G$.
	Thus, $\eta_p (G)\le n-3$.

	\begin{figure}[ht]
		\begin{center}
			\includegraphics[width=0.7\textwidth]{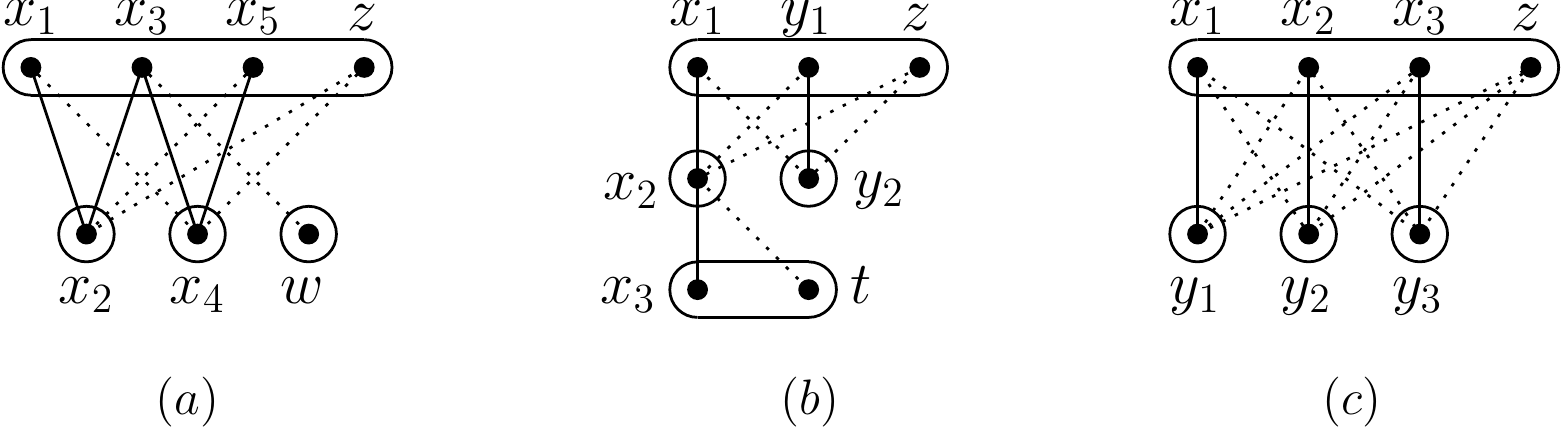}
			\caption{Solid (resp. dotted) lines mean adjacent (resp. non-adjacent) vertices in $\overline{G}$.
			}
			\label{fig1}
		\end{center}
	\end{figure}

	\vspace{.4cm}\noindent \textbf{Case 2}: $\overline{G}$ has at least two non-trivial components and one of them has order at least  $3$.
	If there is only one component of order $\ge 3$, say $C_1$, then there is at least a component of order $2$, say $C_2$.
	Otherwise, there are two components, say $C_1$ and $C_2$, of order at least $3$.
	In both cases,  we may assume that  $x_1,x_2,x_3$ are vertices of $C_1$ and $y_1,y_2$ are vertices of $C_2$, such that $x_1x_2\in E(\overline{G})$, $x_2x_3\in E(\overline{G})$, $y_1y_2\in E(\overline{G})$.
	Since $n\ge 7$, we may assume that there are two more vertices $z$ and $t$ such that at least one of them, say $z$, is not adjacent to $y_2$ in $\overline{G}$.
	
		Consider the partition:
	$$\Pi =
	\{ \{ x_1,y_1,z\} , \{ x_3,t\}\} \cup
	\{ \{ v \} : v\notin  \{  x_1,x_3,y_1,t, z \} \}.
	$$
	We claim that $\Pi$ is an RD-partition of $G$ (see Figure~\ref{fig1} (b)). Indeed,
	recall that two vertices are at distance 2 in $G$ whenever they are adjacent in $\overline{G}$, and they are at distance $1$ in $G$ whenever they are not adjacent in $\overline{G}$. Hence,
	if $S_1=\{ x_2 \}$ and $S_2=\{ y_2 \}$, then
	$r(x_1|\Pi)=(2,1,\dots)$, $r(y_1|\Pi)=(1,2,\dots)$, $r(z|\Pi)=(1,1,\dots)$, and  $r(x_3|\Pi)=(2,\dots)$, $r(t|\Pi)=(1,\dots)$.
	Therefore, $\Pi$ is an RD-partition of $G$ and $\eta_p (G)\le n-3$.

	\vspace{.4cm}\noindent \textbf{Case 3}: All non-trivial components of $\overline{G}$ have order $2$.
	Then,
	$\overline{G}$ has at least 3  components that are copies of $K_2$.
	Let $\{x_i,y_i\}$, for $i=1,2,3$,
	be the vertices of three of these copies, and
	let $z$ be a vertex not belonging to them.
	Then,
	$$\Pi =
	\{ \{ x_1,x_2,x_3,z\} \} \cup
	\{ \{ v \} : v\not= x_1,x_2,x_3,z \}
	$$
	is an RD-partition of $G$ (see Figure~\ref{fig1} (c)).
	Indeed,  if $S_1=\{ y_1 \}$, $S_2=\{ y_2 \}$ and $S_3=\{ y_3 \}$, then
	$r(x_1|\Pi)=(2,1,1,\dots)$, $r(x_2|\Pi)=(1,2,1,\dots)$, $r(x_3|\Pi)=(1,1,2,\dots)$  and  $r(z|\Pi)=(1,1,1,\dots)$.
	Therefore, $\eta_p (G)\le n-3$.
\end{proof}

\begin{prop}\label{prop.diametro2deltaN3}
Let $G$ be  a graph of order {\color{black}$n\ge 7$}, diameter $2$ and minimum degree   at least $n-3$. If  $\eta_p(G)=n-2$, then  $G\in \{  H_1,H_2,H_{3},H_{4},H_{5}\}$ (see Figure \ref{d2deltan-3}).
\end{prop}
\begin{proof}
Let $\Omega\subseteq V(G)$ be the set of vertices of $G$ of degree $n-1$, which according to
Lemma \ref{lemma22} contains at least $n-4$ vertices.
We distinguish cases depending on the cardinality of $\Omega$.

\begin{figure}[!t]
\begin{center}
\includegraphics[width=0.96\textwidth]{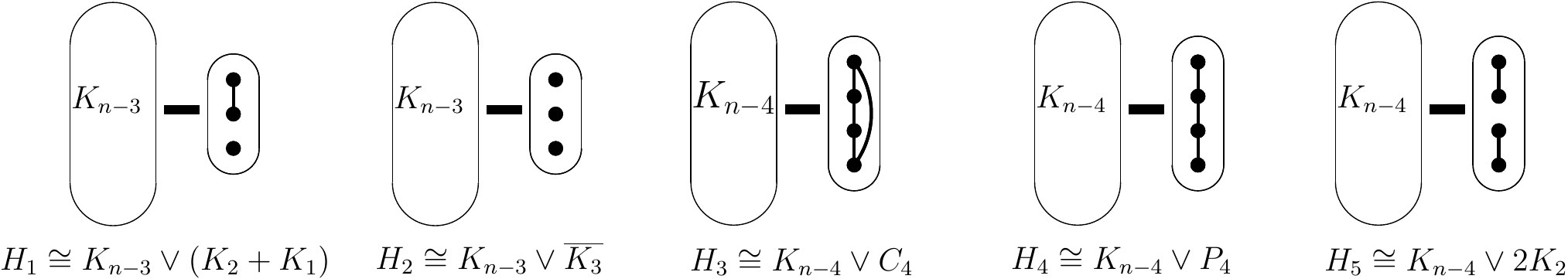}
\caption{Graphs of order $n\ge7$, diameter ${\rm diam}(G)=2$ and minimum degree $\delta(G)\ge n-3$ such that $\eta(G)=n-2$.}
\label{d2deltan-3}
\end{center}
\end{figure}

\noindent \textbf{Case 1}:  $|\Omega|\ge n-2$.
If $|\Omega|= n$, then $G\cong K_n$ and thus $\eta_p(G)=n$.
Case  $|\Omega|=n-1$ is not possible.
If $|\Omega|=n-2$, then $G\cong K_{n-2} \vee \overline{K_2}$, and according to Theorem \ref{etann-1}(2), $\eta_p(G)=n-1$.

\vspace{.2cm}
\noindent \textbf{Case 2}: $|\Omega|=n-3$.
Let $F$ be the subgraph of order 3 induced by $V(G)\setminus \Omega$, i.e., $F=G[V(G)\setminus \Omega]$. Notice that  $|E(F)|\le1$.
If $|E(F)|=1$, then  $G\cong H_1$.
Otherwise, if  $|E(F)|=0$, then $G\cong H_2$.

\vspace{.2cm}
\noindent \textbf{Case 3}: $|\Omega|=n-4$.
Consider the graph of order 4, $F=G[V(G)\setminus \Omega]$.
Note that all vertices of $F$ have degree either 1 or 2.
There are thus three possibilities.
If $F\cong C_4$, then $G\cong H_{3}$.
If $F\cong P_4$, then $G\cong H_{4}$.
If $F\cong 2K_2$, then $G\cong H_{5}$.
\end{proof}

\begin{prop}\label{prop.diametro2delta2}
Let $G$ be a graph of order {$n\ge 7$}, diameter $2$ and minimum degree  at most $n-4$.
If $\eta_p(G)=n-2$, then $G\in \{ H_{6},H_{7},H_8,H_9,H_{10},H_{11},H_{12},H_{13},H_{14}\}$  (see Figure \ref{d2delta12}).
\end{prop}
\begin{proof}
By Lemma \ref{lem.fusion} for $k=3$, we have that $\deg (w)\in \{ 1,2,n-3,n-2,n-1\}$, for every vertex $w\in V(G)$.
Hence, $\delta (G)\le 2$. We distinguish two cases.
\vspace{.1cm}

\noindent \textbf{Case 1}: \emph{There exists a vertex $u$ of degree $2$.}
Consider the subsets $D_1=N(u)=\{ x_1,x_2\}$ and $D_2=\{ v\in V(G): d(u,v)=2\}=V(G)\setminus N[u]$,  so that $|D_2|=n-3$.
\vspace{.1cm}


\noindent \textbf{Case 1.1}:
\emph{$G[D_{2}]$ is neither complete nor empty.}
Then, there exist three different vertices $r,s,t\in D_2$ such that $rs\in E(G)$ and $rt\notin E(G)$.
Let $y\in D_2\setminus \{ r,s,t \}$.
We distinguish cases taking into account whether or not  $y$ and $t$ are leaves.

\begin{itemize}

\item \emph{Both $y$ and $t$ are leaves hanging from the same vertex.}
 Assume that they hang from $x_1$. Let
$S_1=\{ u,y\}$ and  $S_2=\{ x_2,s,t \}$.
In such a case, $S_2$ resolves $S_1$, $\{r\}$ resolves the pair $\{ s,t \}$
and $S_1$ resolves the pairs $\{ x_2, s\}$ and $ \{ x_2, t\}$.
Therefore, $\Pi = \{ S_1,S_2  \}\cup \{ \{ w \} : w\notin S_1\cup S_2  \}$  is a resolving partition.
It can be easily checked that  $\Pi $ is also dominating.
Hence,  $\eta_p(G)\le n-3$, a contradiction.

\item \emph{Both $y$ and $t$ are leaves but not hanging from the same vertex, or neither $y$ nor $t$ are leaves.}
%
%
If both $y$ and $t$ are leaves but not hanging from the same vertex,
assume $x_1y\in  E$ and $x_2t\in E$. Let $S_1=\{ x_2,y \}$ and $S_2=\{ x_1,s,t\}$.
If neither $y$ nor $t$ are leaves
and $N(t)\not= \{ s,x_1 \}$, let $S_1=\{ x_2,y\}$ and $S_2=\{ x_1,s,t \}$.
If neither $y$ nor $t$ are leaves
and $N(t) = \{ s,x_1 \}$, let $S_1=\{ x_1,y\}$ and $S_2=\{ x_2,s,t \}$.
In all these cases,
$\{u\}$ resolves $S_1$,
$\{r\}$ resolves  $\{ s,t \}$,
and $\{u\}$ resolves any other pair from $S_2$.
Hence, $\Pi = \{ S_1,S_2  \}\cup \{ \{ w \} : w\notin S_1\cup S_2 \}$ is a resolving partition of $G$.
It can be easily checked that $\Pi$ is a dominating partition.
Thus, $\eta_p(G)\le n-3$, a contradiction.

\item \emph{Exactly one of the vertices $y$ or $t$ is a leaf.}
We may assume that the leaf hangs from $x_1$.
If $t$ is a leaf, then take $S_1=\{ x_1,y\}$ and  $S_2=\{ x_2, s,t \}$.
If $y$ is a leaf and $N(t)\not= \{x_1,s \} $  then take $S_1=\{ x_2,y\}$ and $S_2=\{ x_1,s,t\}$.
In both cases, $\{ r \}$ resolves  $\{ s,t \}$ and $\{ u \}$ resolves any other pair in either $S_1$ or $S_2$.
If $y$ is a leaf and $N(t)= \{x_1,s\}$  then take $S_1=\{ u,y\}$ and $S_2=\{ x_2,s,t\}$.
Then, $\{r\}$ resolves the pair $\{ s, t \}$, $S_1$ resolves the other pairs from $S_2$; and $S_2$ resolves $S_1$.
In all cases, $\Pi = \{ S_1,S_2  \}\cup \{ \{ w \} : w\notin S_1\cup S_2 \}$ is dominating partition.
Thus, $\eta_p(G)\le n-3$, a contradiction.
\end{itemize}

\begin{figure}[!ht]
\begin{center}
\includegraphics[width=0.8\textwidth]{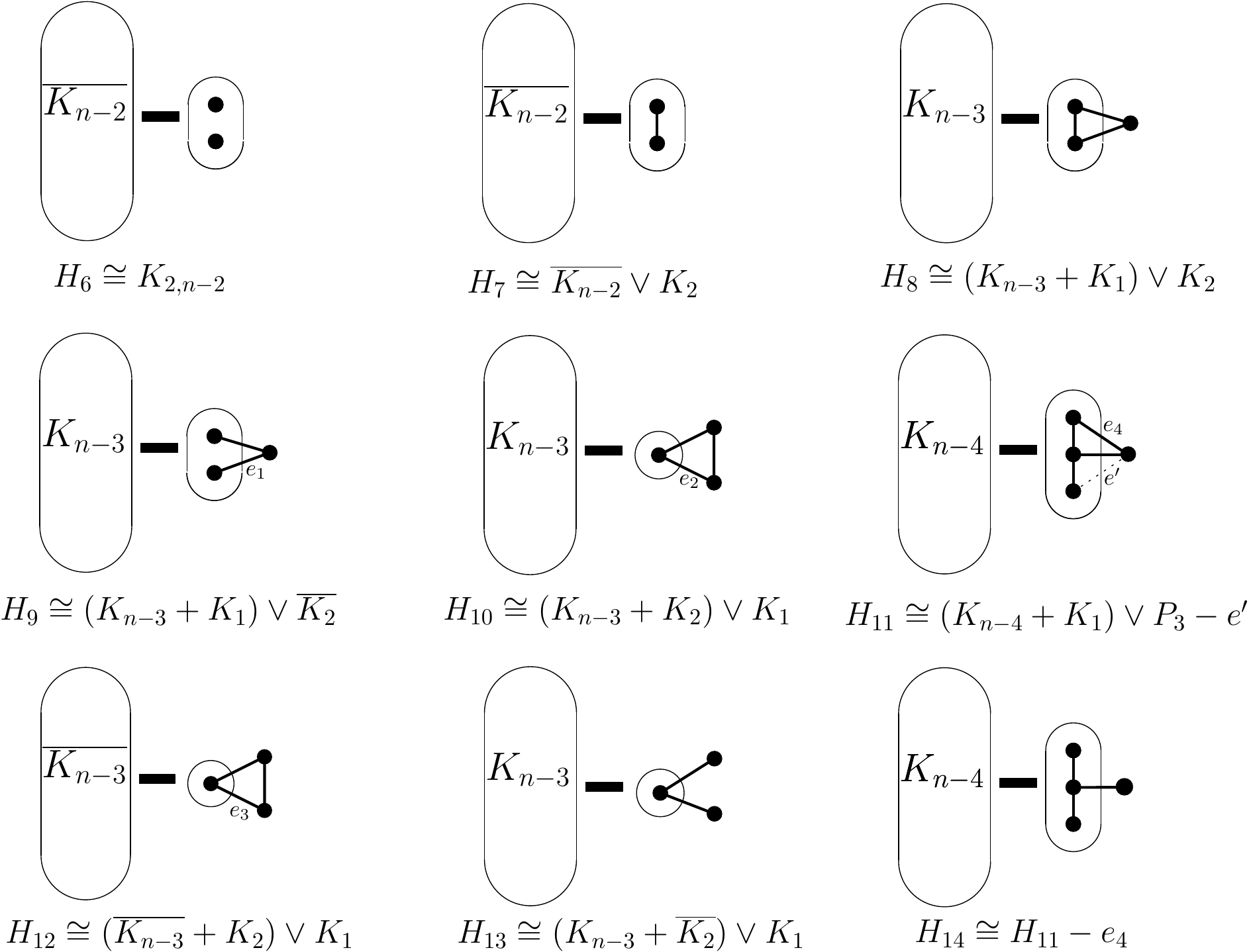}
\caption{Graphs of order $n\ge7$, diameter ${\rm diam}(G)=2$ and minimum degree $1\le\delta(G)\le 2$ such that $\eta(G)=n-2$.}
\label{d2delta12}
\end{center}
\end{figure}

\vspace{.2cm}		
\noindent  \textbf{Case 1.2}:  \emph{$G[D_{2}]$ is either complete or empty.}
Assume that $\deg(x_1) \le \deg(x_2)$.
Consider the subsets
$N_1=N(x_1)\cap D_2$ and $N_2=N(x_2)\cap D_2$. Observe that $N_1\cup N_2=D_2$, and the sets $N_1\setminus N_2$,  $N_1\cap N_2$ and $N_2\setminus N_1$ are pairwise disjoint.
Besides, $|N_2\setminus N_1|\ge |N_1\setminus N_ 2|$ because we have assumed $\deg (x_2)\ge \deg (x_1)$.
Notice also that $\deg(x_2) \ge \deg(x_1) \ge 2$, as otherwise ${\rm diam}(G)\ge3$.
We distinguish two cases.

\vspace{.2cm}
\noindent \textbf{(1.2.1)}: $\deg(x_1)= 2$.
Thus, $\deg(x_2)\geq ( |D_2|-1)+1\ge n-3$.

\begin{itemize}
\item
If $x_1x_2\in E$, then $N_1=\emptyset$ and $D_2=N_2$.
If $G[D_{2}]\cong K_{n-3}$, then $G\cong H_{10}$.
If $G[D_{2}]\cong \overline{K_{n-3}}$, then $G\cong H_{12}$.

\item
If $x_1x_2\notin E$, then
$|N_1|=1$  and $|N_2\setminus N_1|=n-4\ge 3$. If $G[D_{2}]\cong \overline{K_{n-3}}$,  then
${\rm diam}(G)\ge 3$. Hence, $G[D_{2}]\cong K_{n-3}$.
Consider $y \in N_1$ and $z_1, z_2 \in N_2\setminus N_1$.
Let $S_1=\{u, x_1\}$, $S_2=\{x_2, z_2\}$  and $S_3=\{y, z_1\}$
and consider the partition $\Pi = \{ S_1,S_2,S_3 \}\cup \{ \{ w \} : w\notin S_1\cup S_2 \cup S_3 \}$.
Then, $S_1$ resolves both $S_2$ and $S_3$; and $S_3$ resolves $S_1$. Moreover, $\Pi$ is a dominating
partition of $G$. Thus, $\eta_p(G)\le n-3$, a contradiction.

\end{itemize}

\vspace{.2cm}
\noindent \textbf{(1.2.2)}: $\deg(x_1)\ge n-3$.
Hence, $\deg(x_2) \ge \deg(x_1) \ge n-3$.
In such a case, $|N_1|\ge n-5$ and $|N_2|\ge n-5$, and so $n-7\le |N_1\cap N_2|\le n-3$.
We distinguish cases depending on the cardinality of $ |N_1\cap N_2|$.

\begin{itemize}

\item $ |N_1\cap N_2| = n-3$.
Then, $N_1=N_2=V(G)\setminus N[u]$.
If $x_1x_2\in E$, then $G\cong H_8$ if $G[D_{2}]\cong K_{n-3}$, and $G\cong H_7$ if $G[D_{2}]\cong \overline{K_{n-3}}$.
If $x_1x_2\notin E$, then $G\cong H_9$ if $G[D_{2}]\cong K_{n-3}$, and $G\cong H_6$ if $G[D_{2}]\cong \overline{K_{n-3}}$.

\item $ |N_1\cap N_2| = n-4$.
Then, $|N_2\setminus N_1|+|N_1\setminus N_2|=1$.
Thus,  $|N_2\setminus N_1|=1$, $|N_1\setminus N_2|=0$ and $ |N_1\cap N_2| \ge 3$.
If $G[D_{2}]\cong \overline{K_{n-3}}$, then ${\rm diam} (G)\ge 3$, a contradiction.
If $G[D_{2}]\cong K_{n-3}$ and $x_1x_2\in E$, then $G\cong H_{11}$.
If $G[D_{2}]\cong K_{n-3}$ and $x_1x_2\notin E$, then let
$y_1,y_2,y_3\in N_1\cap N_2$ and let $z\in N_2\setminus N_1$.
Consider $S_1=\{ u, y_1\}$, $S_2=\{ x_2, y_2\}$, $S_3=\{ z, y_3\}$ and
let  $\Pi = \{ S_1, S_2, S_3 \}\cup \{ \{ w \} : w\notin S_1\cup S_2 \cup S_3 \}$.
Then, $\{x_1\}$ resolves both $S_2$ and $S_3$, and $S_3$ resolves $S_1$.
It is easy to check that it is a dominating partition. Therefore, $\eta_p(G)\le n-3$, a contradiction.

\item $ |N_1\cap N_2| = n-5$.
Then, $|N_2\setminus N_1|+|N_1\setminus N_2|=2$ and $ |N_1\cap N_2| \ge 2$.
Let $y_1,y_2\in (N_2\setminus N_1)\cup (N_1\setminus N_2)$ and $z_1,z_2\in N_1\cap N_2$,
and let  $S_1=\{ y_1,z_1\}$, $S_2=\{ y_2,z_2\}$ and $S_3=\{ u,x_1 \}$.
Then,
$\Pi = \{ S_1, S_2, S_3 \}\cup \{ \{ w \} : w\notin S_1\cup S_2 \cup S_3 \}$ is an RD-partition of $G$.
Indeed, $S_1$ resolves $S_3$ and, for $i\in\{ 1,2\}$,
$S_i$ is resolved by  $S_1$ if  $y_i\in N_2\setminus N_1$ and
$S_i$ is resolved by $\{ x_2 \}$ if $y_i\in N_1\setminus N_2$.
Besides, $\Pi$ is dominating.
Hence, $\eta_p(G)\le n-3$, a contradiction.

\item $ |N_1\cap N_2| \in \{ n-6, n-7 \}$.
In such a case, $|N_2\setminus N_1|+|N_1\setminus N_2|\in \{ 3,4\}$.
Since $|N_2\setminus N_1|\ge |N_1\setminus N_2|$, we have $|N_2\setminus N_1|\ge 2$.
Since $ \deg (x_1)\ge n-3$, we have $|N_1|\ge n-5\ge 2$.
Let $y_1,y_2\in N_1$ and $z_1,z_2\in N_2\setminus N_1$.
If  $S_1=\{ u,x_1\}$, $S_2=\{ y_1,z_1\}$ and $S_3=\{ y_2,z_2\}$,
and   $\Pi = \{ S_1,S_2,S_3 \}\cup \{ \{ w \} : w\notin S_1\cup S_2 \cup S_3 \}$,
then $S_1$ resolves both $S_2$ and $S_3$, and $S_2$ resolves $S_1$.
Moreover, $\Pi$ is a dominating partition.
Therefore, $\eta_p(G)\le n-3$, a contradiction.
\end{itemize}

\vspace{.2cm}
\noindent \textbf{Case 2}: \emph{There exists at least one vertex $u$ of degree $1$ and there is no vertex of degree 2.}
Since ${\rm diam}(G)=2$,  the neighbor $v$ of $u$ satisfies $\deg(v)=n-1$.
Let $\Omega$ be the set of vertices different from $v$ that are not leaves.
Notice that there are at most two vertices of degree 1 in $G$, as otherwise all vertices in $\Omega$ would have degree between  $3$ and  $n-4$, contradicting the assumption made at the beginning of the proof.

If there are exactly two vertices of degree 1, then $|\Omega|=n-3$.
In such a case, as for every vertex $w\in\Omega$, $\deg(w)\ge n-3$,  $\Omega$ induces a complete graph in $G$, and hence  $G\cong H_{13}$.

Suppose next that $u$ is the only vertex of degree 1, which means that  $\Omega$ contains $n-2$ vertices,
all of them  of degree $n-3$ or $n-2$.
Consider the (non-necessarily connected) graph   $J=\overline{G[\Omega]}$.
Certainly,  $J$ has $n-2$ vertices, all of them  of degree either $0$ or $1$.
Let $L$ denote the set of vertices of degree $1$ in $J$.
Observe that  the cardinality of $L$ must be even.
We distinguish three cases.
\begin{itemize}

\item If $|L|=0$, then $G\cong K_1	\vee (K_1 +K_{n-2})$, and by
Theorem~\ref{etann-1} we have $\eta_p(G)=n-1$,
a contradiction.

\item If $|L|=2$, then $G\cong H_{14}$.

\item If $|L|\ge 4$,
let  $\{x_1,x_2,x_3,x_4\}\subseteq L$ such that $x_1x_2$ and $x_3x_4$ are edges of $J$, and let $y\in \Omega \setminus \{x_1,x_2,x_3,x_4\}$.
Consider the partition $ \Pi =\{ S_1,S_2\} \cup \{ \{ w \} : w\notin S_1\cup S_2 \}$,
where  $S_1= \{v, x_1\}$, $S_2=\{u, x_3, y\}$.
Observe that $\{x_2\}$ resolves $S_1$, $\{u, x_3\}$ and  $\{u,y\}$,
and  $\{x_4\}$ resolves $\{x_3, y\}$.
Besides, $\Pi$ a is dominating partition.
Therefore, $\eta_p(G)\le n-3$, a contradiction.
\end{itemize}
\vspace{-.78cm}\end{proof}

\subsection{Case diameter 3}

We consider the case  $\eta_p(G)= n-2$ and ${\rm diam}(G)=3$.

\begin{prop}\label{prop.diametro3}
Let $G$ be a graph of order $n\ge 7$ and diameter $3$.
 If $\eta_p(G)=n-2$, then  $G \in \{  H_{15}, H_{16}, H_{17}\}$ (see Figure \ref{d3}).
\end{prop}
\begin{proof}
By Lemma~\ref{lem.fusion} (case $k=3$), every vertex has degree $1$, $2$, $n-3$, $n-2$ or $n-1$.
Let $u$ and $v$ be two vertices such that $d(u,v)=3$.
In such a case, both $u$ and $v$ have degree at most  $n-3$.

Notice that on the one hand,  it is not possible to have neither $\{ \deg (u), \deg (v) \}=\{ 2, n-3 \}$ nor $\{ \deg (u), \deg (v) \}=\{ n-3 \}$, as otherwise we would have more than $n$ vertices because $N(u)\cap N(v)=\emptyset$, a contradiction.

On the other hand, if $\deg (u)=\deg (v)= 2$, then $\eta_p(G)\le n-3$.
Indeed, let $ux_1x_2v$ be a $(u,v)$-path and let $D_i = \{ z : d(u,z)=i\}$, for $i\in \{1,2,3 \}$.
Since $ |D_1|=2$, we may assume that $D_1=\{ x_1,y_1\}$.
If $|D_2|\ge 2$, let $y_2\in D_2\setminus \{ x_2 \}$.
If $x_1 y_2\in E$, let $S_1=\{ x_1, x_2\}$ and $S_2=\{ y_1,y_2,v\}$.
If $x_1 y_2\notin E$, then $y_1 y_2\in E$, and consider $S_1=\{ y_1,x_2\}$ and $S_2=\{ x_1, y_2, v\}$.
If $|D_2|=1$, then $v$ has a neighbor $z\in D_3$, so that $z$ must be also adjacent to $x_2$.
Let  $S_1=\{ x_1, x_2, v \}$ and $S_2=\{ y_1,z\}$.
In all cases, $\Pi =\{ S_1, S_2 \}\cup \{  \{ w \} : w\notin S_1\cup S_2\}$
is an RD-partition, because it is dominating and $\{ u \}$ resolves both $S_1$ and $S_2$.
Hence,  $\eta_p(G)\le n-3$, a contradiction.

Therefore, we may assume that  $\deg (u)=1$  and that every  vertex at distance $3$ from $u$ has degree $1$, $2$ or $n-3$.
Let $D_i=\{ x\in V(G) : d(u,x)=i \}$, for $i=1,2,3$.
Thus, $|D_1|=1$. Let $D_1=\{ w \}$.
We distinguish cases, depending on the cardinality of $D_3$.
\vspace{2mm}

\begin{figure}[!ht]
\begin{center}
\includegraphics[width=0.65\textwidth]{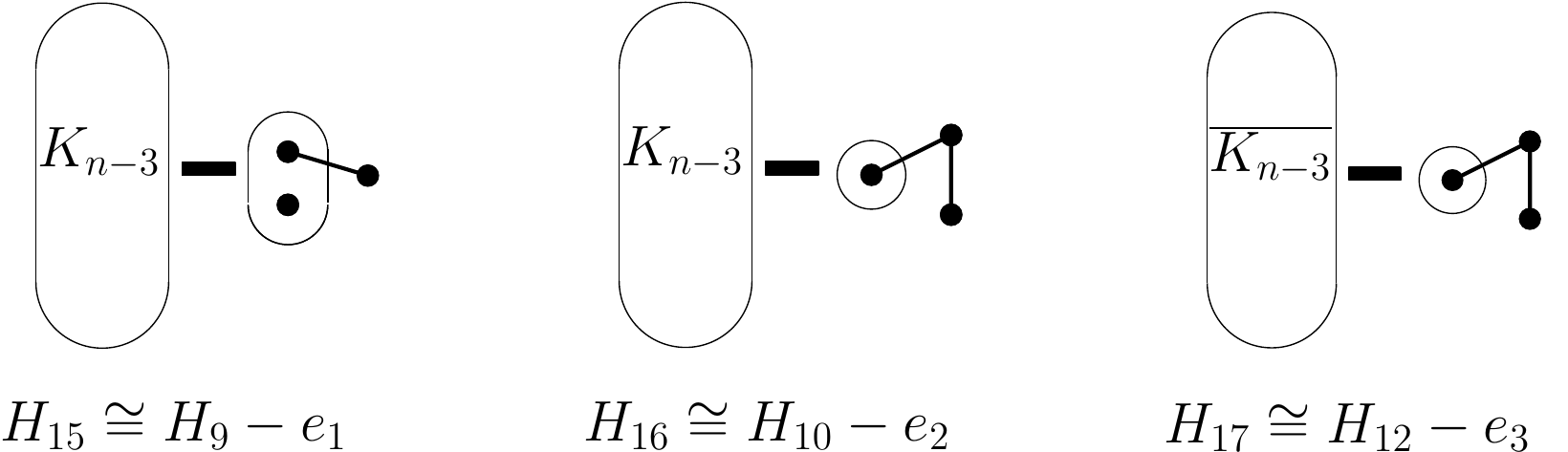}
\caption{Graphs of order $n\ge7$ and diameter $3$ such that $\eta(G)=n-2$.}
\label{d3}
\end{center}
\end{figure}

{\bf Case 1:} $|D_3|\ge 3$.
Then, $\deg (w)\le n-4$, and therefore,  $\deg (w)=2$, $|D_1|=|D_2|=1$ and  $|D_3|=n-3\ge 4$.
Let $x$ be the only vertex in $D_2$.
Notice that every vertex of $D_3$ is adjacent to $x$.
We distinguish cases taking into account the degree of the vertices in $D_3$.

\begin{itemize}

\item \emph{There is a vertex of degree  $n-3$ in $D_3$.}
A vertex in  $D_3$ of degree $n-3$ must be adjacent to all the other vertices of $D_3$.
Therefore,  there is exactly one vertex of degree $n-3$ in $D_3$ or every vertex in $D_3$ has degree $n-3$.
In the last case, that is, if every vertex in $D_3$ has degree $n-3$, then $D_3$ is a clique and $G\cong H_{16}$.
Otherwise, let $y_1$ be the only vertex in $D_3$ of degree $n-3$.
Any other vertex in $D_3$ has degree $2$, since it is adjacent to $x$ and to $y_1$.
Let $y_2,y_3,y_4\in D_3\setminus \{ y_1 \}$.
Consider $S_1=\{ y_1,y_2\}$ and $S_2=\{ w,x,y_3\}$.
Then,
$\Pi =\{ S_1, S_2 \} \cup \{ \{ z \} :  z\notin S_1\cup S_2 \}$
is an RD-partition of $G$.
Indeed, it is dominating partition, $\{u \}$ resolves $S_2$ and $\{ y_4 \}$ resolves $S_1$
(see Figure~\ref{figpropdiam3}(a)).
Thus, $\eta_p(G)\le n-3$, a contradiction.

\item \emph{Every vertex in $D_3$ has degree 1 or 2, and  at least one of them has degree 2.}
Then, $G[D_3]$ contains at least a copy of $K_2$.
Let $y_1$ and $y_2$ be the vertices of such a copy of $K_2$, and take $y_3\in D_3\setminus \{ y_1,y_2 \}$.
Consider $S_1=\{ w,y_1 \}$, $S_2=\{ x,y_2\}$ and $S_3=\{ u,y_3 \}$.
It is straightforward to check that
$\Pi =\{ S_1, S_2, S_3 \} \cup \{ \{ z \} :  z\notin S_1\cup S_2 \cup S_3 \}$
is an RD-partition of $G$ (see Figure~\ref{figpropdiam3}(b)), and thus $\eta_p(G)\le n-3$, a contradiction.

\item \emph{Every vertex in $D_3$ has degree 1.}
Then, $D_3$ induces an empty graph and $G\cong H_{17}$.

\end{itemize}

\begin{figure}[ht]
\begin{center}
\includegraphics[width=0.85\textwidth]{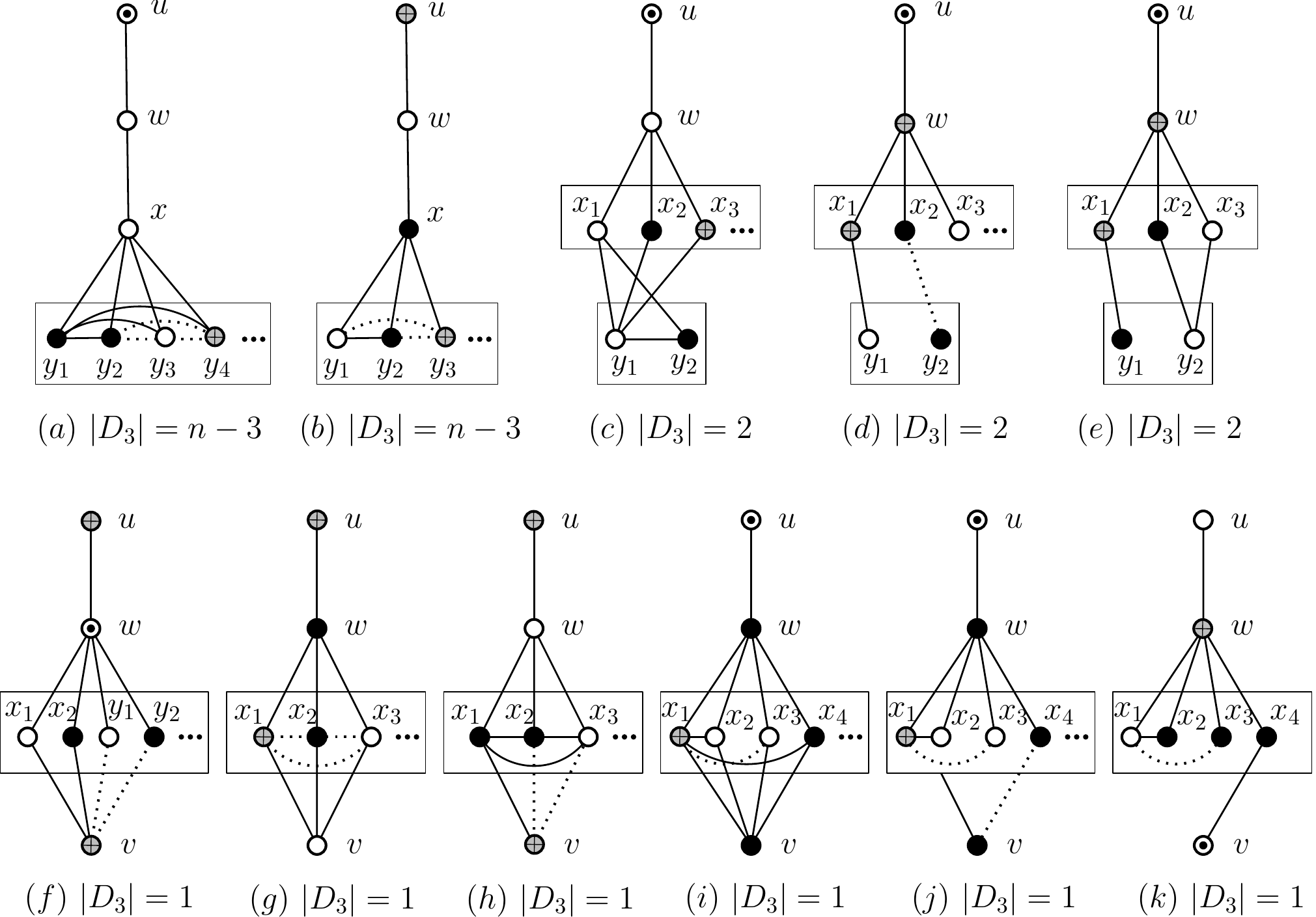}
\caption{Solid (resp. dotted) lines mean adjacent (resp. non-adjacent) vertices.
Vertices with the same "color" belong to the same part.}
\label{figpropdiam3}
\end{center}
\end{figure}

{\bf Case 2:} $|D_3|=2$.
Then, $|D_2|=n-4$.
Let  $D_3=\{ y_1,y_2 \}$.
Recall that both $y_1$ and $y_2$ have at least a neighbor in $D_2$.
We distinguish cases taking into account the degree of the vertices in $D_3$.

\begin{itemize}

\item \emph{There is a vertex of degree $n-3$ in $D_3$}.
We may assume that this vertex is $y_1$,  and it must be adjacent to $y_2$ and to all vertices in $D_2$.
So, there is a vertex $x_1\in D_2$ adjacent to both $y_1$ and  $y_2$.
Let $x_2\in D_2\setminus \{ x_1 \}$ and consider
$S_1= \{ w,x_1,y_1 \}$ and $S_2= \{ x_2,y_2\}$
Then,
$\Pi =\{ S_1, S_2\}\cup \{ \{ z \} :  z\notin S_1\cup S_2 \}$
is a dominating partition, and $\{ u\}$ resolves both $S_1$ and $S_2$
 (see Figure~\ref{figpropdiam3}(c)).
Hence,  $\eta_p(G)\le n-3$, a contradiction.

\item \emph{Both vertices in $D_3$ have degree 1 or 2}.
Let $x_1\in D_2$ be a neighbor of $y_1$.

If there exists a vertex $x_2\in D_2\setminus \{ x_1\}$
not adjacent to $y_2$, let $x_3\in D_2\setminus \{ x_1 ,x_2 \}$.
Consider
$S_1= \{ w,x_1 \}$, $S_2= \{ x_2,y_2\}$ and $S_3=\{ x_3,y_1\} $.
Then, $\Pi =\{  S_1, S_2, S_3 \}\cup \{ \{ z \} :  z\notin S_1\cup S_2\cup S_3 \}$
is a dominating partition and $\{ u \}$ resolves $S_1$, $S_2$ and $S_3$ (see Figure~\ref{figpropdiam3}(d)).
Therefore,  $\eta_p(G)\le n-3$, a contradiction.

If all vertices in  $D_2\setminus \{ x_1\}$ are adjacent to $y_2$, then $\deg(y_2)\ge n-5$,
with means that $2=\deg(y_2)=n-5$  and thus $n=7$.
Let $D_2=\{ x_1, x_2, x_3\}$ and consider
$S_1= \{ w,x_1 \}$, $S_2= \{ x_2,y_1\}$ and $S_3=\{ x_3,y_2\} $.
Then, $\Pi =\{  S_1, S_2, S_3 \}\cup \{ \{ z \} :  z\notin S_1\cup S_2\cup S_3 \}$
is a dominating partition and $\{ u \}$ resolves $S_1$, $S_2$ and $S_3$  (see Figure~\ref{figpropdiam3}(e)).
Therefore,  $\eta_p(G)\le n-3$, a contradiction.

\end{itemize}

{\bf Case 3:}  $|D_3|=1$.
Then, $D_3=\{v \}$ and $|D_2|=n-3$.
We distinguish cases taking into account the degree of $v$ and the subgraph induced by $D_2$.
\begin{itemize}

\item  $\deg (v)=2$. Let $x_1$ and $x_2$ be the two neighbors of $v$, and take $y_1,y_2\in D_2\setminus \{x_1,x_2\}$.
Let $S_1=\{ u,v\}$, $S_2=\{ x_1,y_1\}$ and $S_3=\{ x_2,y_2\}$.
Then,
$\Pi =\{ S_1, S_2, S_3   \} \cup \{ \{ z \} :  z\notin S_1\cup S_2 \cup S_3  \}$
is dominating partition such that $\{ w \}$ resolves $S_1$,  and $S_1$ resolves both $S_2$ and $S_3$
(see Figure~\ref{figpropdiam3}(f)), implying that  $\eta_p(G)\le n-3$, a contradiction.

\item  \emph{$\deg (v)\in \{1,n-3\}$ and $D_2$ induces an empty graph.}

If  $\deg (v)=n-3$, let $x_1,x_2,x_3\in D_2$ and let
$S_1=\{ u,x_1\}$,  $S_2=\{ w,x_2\}$ and $S_3=\{ v,x_3\}$.
Then,
$\Pi =\{ S_1, S_2, S_3  \} \cup \{ \{ z \} :   z\notin S_1\cup S_2 \cup S_3  \}$
is a dominating partition such that $S_1$ resolves both $S_2$ and $S_3$, and $S_3$ resolves $S_1$
(see Figure~\ref{figpropdiam3}(g)), implying that  $\eta_p(G)\le n-3$, a contradiction.

If $\deg (v)=1$, then $G\cong H_{17}$.

\item   \emph{$\deg (v)\in \{1,n-3\}$ and $D_2$ induces a complete  graph.}

If $\deg (v)=n-3$, then $G\cong H_{15}$.

If $\deg (v)=1$, let $x_1\in D_2$ be the neighbor of $v$ and $x_2,x_3\in D_2\setminus \{ x_1 \}$.
Consider
$S_1=\{ u,v\}$,  $S_2=\{ w,x_3\}$ and $S_3=\{ x_1,x_2\}$.
Then,
$\Pi =\{ S_1, S_2, S_3  \} \cup \{ \{ z \} :   z\notin S_1\cup S_2 \cup S_3  \}$
is a dominating partition such that  $S_1$ resolves both $S_2$ and $S_3$, and $S_3$ resolves $S_1$
(see Figure~\ref{figpropdiam3}(h)), implying that  $\eta_p(G)\le n-3$, a contradiction.

\item   \emph{$\deg (v)\in \{1,n-3\}$ and $D_2$ induces neither a complete, nor an empty graph.}

In  that case, there exist vertices $x_1,x_2,x_3\in D_2$ such that $x_1x_2\in E(G)$ and  $x_1x_3\notin E(G)$.

If $\deg (v)=n-3$, then $\deg (x_1)\ge 3$, and thus, $\deg (x_1)\ge n-3$.
Hence, $x_1$ must be adjacent to any other vertex in $D_2$ different from $x_3$.
Let $x_4\in D_2\setminus \{ x_1,x_2,x_3\}$ and
consider
$S_1=\{ w,x_4,v\}$ and  $S_2=\{x_2,x_3\}$
Then,
$\Pi =\{ S_1, S_2  \} \cup \{ \{ z \} :   z\notin S_1\cup S_2 \}$
is a dominating partition such that  $\{ u \}$ resolves $S_1$ and $\{ x_1 \}$ resolves $S_2$
(see Figure~\ref{figpropdiam3}(i)), implying that  $\eta_p(G)\le n-3$, a contradiction.

Finally, suppose that $\deg (v)=1$.
If there is a leaf $x$ in $D_2$, then $d(u,v)=d(x,v)=3$.
In such a case, interchanging the role of the vertices $u$ and $v$, the preceding cases for $|D_3|\ge 2$ apply and we are done. So, we can assume that any vertex in $D_2$ has degree at least 2.
Suppose that $v$ is not adjacent to some vertex $x_4\in D_2\setminus \{ x_1,x_2,x_3\}$.
Notice that such a vertex exists whenever $n\ge 8$, because $D_2$ has at least 5 vertices.
Let $S_1=\{ w,x_4,v\}$ and $S_2=\{x_2,x_3\}$.
Then,
$\Pi =\{ S_1, S_2  \} \cup \{ \{ z \} :   z\notin S_1\cup S_2 \}$
is a dominating partition such that
$\{ u \}$ resolves $S_1$ and $\{ x_1 \}$ resolves $S_2$.
Therefore, $\Pi$ is an RD-partition of $G$  (see Figure~\ref{figpropdiam3}(j)), and so $\eta_p(G)\le n-3$, a contradiction.

Finally, if $n=7$ and the only vertex $x_4\in D_2\setminus \{ x_1,x_2,x_3\}$ is  adjacent to $v$,
take $S_1=\{ x_2,x_3,x_4 \}$ and $S_2= \{ u,x_1 \}$.
Then,
$\Pi =\{ S_1, S_2  \} \cup \{ \{ z \} :   z\notin S_1\cup S_2 \}$
is a dominating partition such that
$\{ v \}$ resolves  both $\{ x_2,x_4 \}$  and $\{x_3,x_4\}$; $S_2$ resolves $\{ x_2,x_3 \}$;
and $S_1$ resolves $S_2$. Therefore, $\Pi$
is an RD-partition of $G$  (see Figure~\ref{figpropdiam3}(k)), and so $\eta_p(G)\le n-3$, a contradiction.
\end{itemize}
\vspace{-.7cm}\end{proof}


As a straightforward consequence of Propositions \ref{propida}, \ref{prop.diametro2deltaN3}, \ref{prop.diametro2delta2} and \ref{prop.diametro3}, the following result is obtained.

\begin{theorem} \label{etan-2}
If $G$ is a graph of order {\color{black} $n\ge 7$},
then  $\eta_p(G)=n-2$ if and only if $G\in \Lambda_n$ (see Figure~\ref{17grafos}).
%
\end{theorem}

The solution for $\beta_p(G)=n-2$ is also almost immediately derived.

\vspace{.2cm}
\begin{theorem}\label{betan-2}
If $G$ is a graph of order {\color{black} $n\ge 7$},
then $\beta_p(G)=n-2$ if and only if $G\in \Lambda_n\setminus\{H_{12},H_{17}\}$.
\end{theorem}
\begin{proof}
If $G\in \Lambda_n\setminus\{H_{12},H_{17}\}$ then, according to Proposition \ref{propida}, $\beta_p(G)=n-2$.
\newline Conversely, let $G$ be a graph of order $n\ge 7$ such that $\beta_p(G)=n-2$.
Thus, $\eta_p(G)=n-2$, since by Theorem \ref{mdpd} and Theorem \ref{etann-1} we know that $\beta_p(G)\ge n-1$ if and only if
$\eta_p(G)\ge n-1$.
Hence, by Theorem \ref{etan-2}, we derive that $G \in \Lambda_n$.
Finally, $\beta_p(G)= n-3$ if $G\in \{ H_{12},H_{17} \}$.
Indeed, in such a case, $\beta_p(G)\ge n-3$, because $G$ contains a twin set of cardinality $n-3$, and a resolving partition of cardinality $n-3$ for $H_{12}$ and $H_{17}$ is shown in Figure \ref{h12h17}.
\end{proof}

\begin{figure}[ht]
\begin{center}
\includegraphics[width=0.52\textwidth]{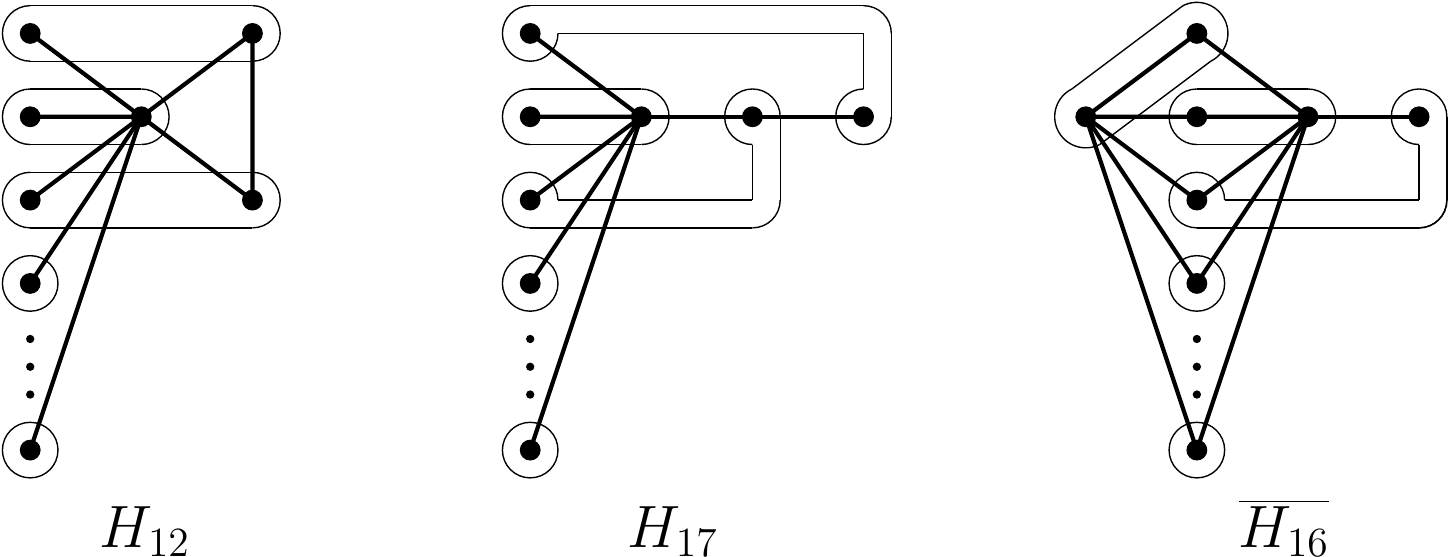}
\caption{Resolving partitions of cardinality $n-3$ of $H_{12}$, $H_{17}$ and $\overline{H_{16}}$.}
\label{h12h17}
\end{center}
\end{figure}

\vspace{.2cm}
\begin{remark}\label{wrong}
Theorem \ref{betan-2} corrects an  inaccurate result shown in \cite {tom08} (Theorem 3.2).
\end{remark}

\vspace{.3cm}
A graph $G$ is called \emph{doubly-connected} if both $G$ and its complement $\overline{G}$ are connected.
We finally show a couple of Nordhaus-Gaddum-type results, which are  a straightforward consequence of Theorems \ref{etan-2} and \ref{betan-2}.

\begin{theorem} \label{ngresult3}
If $G$ is a doubly-connected graph of order $n\ge3$,
then
\begin{enumerate}

\item[\rm(1)] $6 \le \eta_p(G)+\eta_p(\overline{G}) \le 2n-4$.

\item[\rm(2)] The equality $\eta_p(G)+\eta_p(\overline{G})=6$  is attained, among others, by $P_4$ and $C_5$.

\item[\rm(3)] If $n\ge7$, then $\eta_p(G)+\eta_p(\overline{G})=2n-4$ if and only if  $G\in \{H_{15},H_{17}\}$.

\end{enumerate}
\end{theorem}
\begin{proof}
\begin{enumerate}

\item[\rm(1)]
Note that $\eta_p(G)=2$ if and only if $G\cong P_2$, but in this case $\overline{G}$ is not connected. Thus, if $G$ is a doubly-connected graph of order $n$, then $\eta_p(G)\geq 3$ and $\eta_p(\overline{G})\geq 3$, and the lower bound holds.
On other hand, by Theorem~\ref{etann-1},  if $\eta_p(G)\ge n-1$, then $\overline{G}$ is not connected. Thus, $\eta_p(G)+\eta_p(\overline{G}) \le 2n-4$.

\item[\rm(2)]
We know that $\overline{P_4}=P_4$ and $\overline{C_5}= C_5$, and it is easily verified that $\eta_p(P_4)=3$ and $\eta_p(C_5)=3$. Hence, $P_4$ and $C_5$ satisfy the given equality.

\item[\rm(3)]
Finally, a doubly-connected graph $G$ of order at least 7 attaining the upper bound must satisfy $\eta_p(G)=\eta_p(\overline{G})=n-2$.
Therefore, the equality $\eta_p(G)+\eta_p(\overline{G})=2n-4$ is attained if and only if $\{G,\overline{G}\}\subseteq\{  H_1,\dots ,H_{17}\}$ (see Theorem \ref{etan-2}).
It is easy to check that this is satisfied if and only if $G\in \{ H_{15},H_{17}\}$ (observe that $\overline{H_{15}}=H_{17}$).
\end{enumerate}
\end{proof}

\begin{theorem} \label{ngresult4}
If $G$ is a doubly-connected graph of order $n\ge3$,
then

\begin{enumerate}

\item[\rm(1)]  $4 \le \beta_p(G)+\beta_p(\overline{G}) \le 2n-5$.

\item[\rm(2)]  $\beta_p(G)+\beta_p(\overline{G})=4$ if and only if $G = P_4$.

\item[\rm(3)]  If $n\ge7$, then $\beta_p(G)+\beta_p(\overline{G})=2n-5$ if and only if $G\in\{ H_{15},H_{16},H_{17}\}$.

\end{enumerate}

\end{theorem}

\begin{proof}
\begin{enumerate}

\item[\rm(1)]  Every graph $G$ of order at least 3 satisfies  $\beta_p(G)\ge 2$. Hence, the lower bound holds.
By Theorem~\ref{mdpd}, if a graph $G$ satisfies $\beta_p(G)\ge n-1$, then $\overline{G}$ is not connected. Therefore, any doubly-connected graph
$G$ satisfies $\beta_p(G)\le n-2$. By Theorem~\ref{betan-2}, the graphs $G$ satisfying $\beta_p(G)=n-2$ are those from $\Lambda_n\setminus\{H_{12},H_{17}\}$. It is easy to check that the only doubly-connected graphs of this set are $H_{15}$ and $H_{16}$. Their complements are $\overline{H_{15}}=H_{17}$, and $\overline{H_{16}}$ is shown in Figure~\ref{h12h17}. On the one hand, we have seen in the proof of Theorem~\ref{betan-2} that $\beta_p(H_{17})=n-3$. On the other hand, we have that  $\beta_p(H_{16})=n-3$.
Indeed, $\beta (H_{16})\le n-3$ because $H_{16}$ has a twin set of cardinality  $n-3$, and a resolving partition of cardinality $n-3$ is given in Figure~\ref{h12h17}.
Hence, $\beta_p(G)+\beta_p(\overline{G})\le 2n-5$ if $G$ is doubly-connected.

\item[\rm(2)]  We know that $\beta_p(G)=2$ if and only if $G$ is the path $P_n$, and $\overline{P_n}$ is a path if and only if $n=4$. Hence,
the equality $\beta_p(G)+\beta_p(\overline{G})=4$ holds if and only if $G\cong P_4$.

\item[\rm(3)]  This equality is satisfied if and only if $G$ is a doubly-connected graph such that $\{ \beta_p(G),\beta_p(\overline{G} \}=\{ n-2,n-3\}$, and as we have seen in the proof of item i), it happens if and only if $G\in\{ H_{15},H_{16},H_{17}\}$.
\end{enumerate}
\end{proof}


\noindent
{\bf Acknowledgement} The authors are thankful to the anonymous referees for their valuable comments and remarks, which helped to improve the presentation of this paper.


\end{document}